\documentclass{article}

\usepackage{amssymb,amsthm}
\usepackage{amsmath}
\usepackage{latexsym}
\usepackage{amssymb}
\usepackage{graphics}
\usepackage[pdftex]{graphicx}
\newtheorem{theorem}{Theorem}
\newtheorem*{theorem*}{Theorem}

\newtheorem{lemma}[theorem]{Lemma}
\newtheorem{corollary}[theorem]{Corollary}
\newtheorem{definition}[theorem]{Definition}
\newtheorem{remark}[theorem]{Remark}
\newtheorem{example}[theorem]{Example}

\newcommand{\R}{{\mathbb R}}
\newcommand{\Z}{{\mathbb Z}}

\renewcommand{\int}{\rm Int}
\newcommand{\E}{\mathbb E}

\newcommand{\PP}{\Bbb {P}}

\newcommand{\p}{{\mathfrak p}}
\newcommand{\var}{{\rm {Var}}}
\newcommand{\cov}{{\rm {Cov}}}

\graphicspath{%
    {converted_graphics/}
    {/}
}
\begin{document}
\title{Large random simplicial complexes, II;\\ the fundamental group.}        
\author{A. Costa and M. Farber}        
\date{August 13, 2015}          
\maketitle

\section{Introduction} 

This paper develops the multi-parameter model of random simplicial complexes initiated in \cite{CF14} and \cite{CF15a}.

One of the main motivations to study random simplicial complexes comes from the theory of large networks. 
Traditionally one models networks by graphs with nodes representing objects and edges representing connections between the objects
\cite{Newman}.
However if we are interested not only in pairwise relations between the objects but also in relations between multiple objects we may 
use the high dimensional simplicial complexes
instead of graphs as mathematical models of networks. 

The mathematical theory of large random simplicial complexes is a new active research area, see \cite{CFK} and \cite{Ksurvey} for surveys.

The multi-parameter model which we discuss here 
allows regimes controlled by a combination of probability parameters associated to various dimensions. This model includes the well-known Linial - Meshulam - Wallach model \cite{LM}, \cite{MW} as an important special case; as another important special case it includes the random simplicial complexes arising as clique complexes of random Erd\H os--R\'enyi graphs, see \cite{Kahle1}, \cite{CF2}.

In the multi-parameter model one starts with a set of $n$ vertices and retains each of them with probability $p_0$; 
on the next step one connects every pair of retained vertices by an edge with probability $p_1$, and then fills in every triangle in the obtained random graph with probability $p_2$, and so on. 
As the result we obtain a random simplicial complex depending on the set of probability parameters 
$$(p_0, p_1, \dots, p_r), \quad 0\le p_i\le 1.$$
The topological and geometric properties of multi-parameter random simplicial complexes depend on the whole set of parameters 
and their thresholds can be understood as subsets of the space of multi-parameter space and not as single numbers as in all the previously studied models. 

In our recent paper \cite{CF15a} 
we described the conditions under which a multi-parameter random simplicial complex is connected and simply connected. 
In \cite{CF15} we showed that the Betti numbers of multi-parameter random simplicial complexes in one specific dimension dominate 
significantly the Betti numbers in all other dimensions. 
In this paper we focus mainly on the properties of fundamental groups of multi-parameter random simplicial complexes, 
which can be viewed as a new class of random groups. 
We describe thresholds for nontrivially and hyperbolicity (in the sense of Gromov) for these groups. 
Besides, we find domains in the multi-parameter space where
these groups have 2-torsion. We also prove that these groups have never odd-prime torsion and their geometric and cohomological dimensions are either 
$0$, $1$, $2$ or $\infty$. Another result presented in this paper states that aspherical 2-dimensional subcomplexes of random complexes satisfy the Whitehead Conjecture, i.e. all their subcomplexes are also aspherical (with probability tending to one).

\begin{figure}[h] 
   \centering
   \includegraphics[width=2.6in]{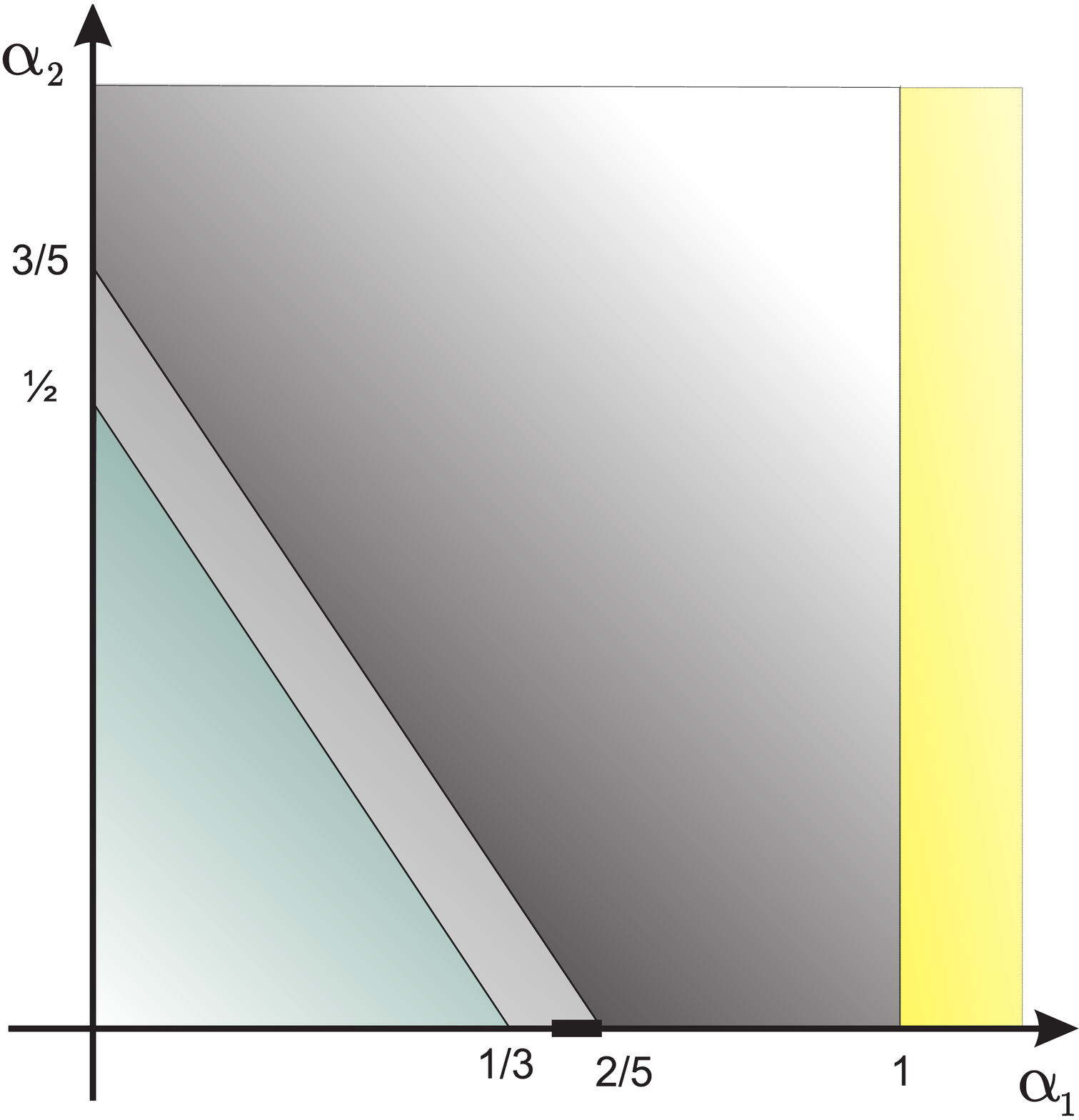} 
   \caption{Areas on the $(\alpha_1, \alpha_2)$-plane corresponding to various properties of the fundamental group: (a) light green - the group is trivial; (b) grey - the group has 2-torsion and is hyperbolic; (c) shaded black (including the horizontal interval $(11/30, 2/5)$ shown in bold) - the group is nontrivial, hyperbolic, its geometric dimension is $\le 2$; 
   (d) yellow - the group is trivial for any choice of the base point.}
   \label{fig:global}
\end{figure}

To make this paper less dependent on \cite{CF15a} we give now a brief description of the multi-parameter model. 
Fix an integer $r\ge 0$ and a sequence $${\mathfrak p}=(p_0, p_1, \dots, p_r)$$ of real numbers satisfying $$0\le  p_i\le 1.$$ Denote  
$q_i=1-p_i.$
We consider the probability space ${\Omega_n^r}$ consisting of all subcomplexes 
$Y\subset \Delta_n$ with $\dim Y\le r,$
where the symbol $\Delta_n^{(r)}$ stands for the $r$-dimensional skeleton of $\Delta_n$, which is defined as the union of all simplexes of dimension $\leq r$. 
The probability function
\begin{eqnarray*}
\PP_{r,\p}: {\Omega_n^r}\to \R
\end{eqnarray*} is given by the formula
\begin{eqnarray}\label{def111}
\PP_{r, \p}(Y) \,  
&=& \, \prod_{i=0}^r p_i^{f_i(Y)}\cdot \prod_{i=0}^r q_i^{e_i(Y)}
\end{eqnarray}
Here $e_i(Y)$ denotes the number of $i$-dimensional external faces of $Y$, see \cite{CF15a} for the definition. 
We use the convention $0^0=1$; in other words, if $p_i=0$ and $f_i(Y)=0$ then the corresponding factor in (\ref{def111}) equals 1.

Some results presented in the paper are illustrated by Figure \ref{fig:global}. There for simplicity we assume that the probability parameters $p_i$ have the form $$p_i=n^{-\alpha_i}$$ where $\alpha_i\ge 0$ are constant. We also assume that $\alpha_0=0$ and hence various properties of the fundamental group are 
described by subsets of the $(\alpha_1, \alpha_2)$-plane since they clearly depend only on $p_1$ and $p_2$. Figure \ref{fig:global} shows domains corresponding to triviality and non-triviality, existence and non-existence of 2-torsion as well as the domain showing when the geometric and cohomological dimension of the fundamental group are $\le 2$.

In area (a) the complex is connected and simply connected, as proven in  \cite{CF15a}. In area (d) the random complex is disconnected and has no cycles and no 2-simplexes, i.e. it is a forest.  
Of course, all these properties are satisfied asymptotically almost surely (a.a.s.), i.e. the limit of the probability that the corresponding property holds tends to 1 as $n\to \infty$.

This research was supported by an EPSRC grant.

\section{Subcomplexes of random complexes}

Let $S$ be a fixed simplicial complex of dimension $\le r$. Similarly to the random graph theory, one may wish to determine conditions when $S$ is embeddable into a random 
simplicial complex $Y\in \Omega_n^r$ with respect to the probability measure $\PP_{r, \p}$, where $\p=(p_0, p_1, \dots, p_r)$. 
In a recent paper \cite{CF14} we discussed the containment problem under a simplifying assumption that $p_i=n^{-\alpha_i}$ where $\alpha_i\ge 0$ are constant. 
In this paper we shall need the containment result without this assumption. Besides the general containment result stated as Theorem \ref{thm1} we also treat in this section the relative containment problem, see Theorem \ref{thm3}. 

\begin{theorem}\label{thm1} Consider a random 
simplicial complex $Y\in \Omega_n^r$ with respect to the measure $\PP_{r, \p}$, where $\p=(p_0, p_1, \dots, p_r)$. Let $S$ be a fixed finite simplicial complex of dimension $\le r$. 

{\rm \bf A}. Suppose that 
\begin{eqnarray}\label{assumption1}
n\cdot\,  \min_{T\subset S} \, \prod_{i=0}^r \, p_i^{\frac{f_i(T)}{f_0(T)}}\to 0.
\end{eqnarray}
Here $T\subset S$ runs over all simplicial subcomplexes of $S$ and $f_i(T)$ denotes the number of $i$-dimensional faces in $T$.  
Then the probability that a random complex $Y\in \Omega_n^r$ contains a simplicial subcomplex isomorphic to $S$ tends to zero as $n\to \infty$. 

{\rm \bf B}. Suppose that for any nonempty subcomplex $T\subset S$ one has
$$n^{f_0(T)}\cdot \prod_{i=0}^r p_i^{f_i(T)} \to \infty.$$
Then a random complex $Y\in \Omega_n^r$ contains a subcomplex isomorphic to $S$ with probability tending to 1. 
\end{theorem}

\begin{proof}
Let us start with the statement {\rm \bf A}. 
For $T\subset S$ let $X_T:\Omega_n^r\to \R$ denote the random variable counting the number of embeddings of $T$ into a random simplicial complex $Y\in \Omega_n^r$. One has 
$$\E(X_T) = \binom n {f_0(T)} \cdot f_0(T)!\cdot \prod_{i=0}^r p_i^{f_i(T)}.$$
Denoting by $\PP_{r, \p}(Y\supset S)$ the probability that $Y$ contains a subcomplex isomorphic to $S$, we have 
 (for $n$ fixed)
\begin{eqnarray}
\PP_{r, \p}(Y\supset S) &\le& \min_{T\subset S}\PP_{r, \p}(Y\supset T)\le \min_{T\subset S} \E(X_T)\nonumber\\ 
& \le & \min_{T\subset S}n^{f_0(T)}\cdot \prod_{i=0}^r p_i^{f_i(T)}\label{esti}
.\end{eqnarray}
By (\ref{assumption1}) for any $\epsilon \in (0,1)$ there exists $N$ such that for any $n\ge N$ there exists $T_n\subset S$, $T\not=\emptyset$, such that 
$$n\cdot \prod_{i=0}^r p_i^{\frac{f_i(T_n)}{f_0(T_n)}} <\epsilon<1.$$
Then 
$$n^{f_0(T_n)}\cdot \prod_{i=0}^r p_i^{f_i(T_n)} <\epsilon^{f_0(T_n)}<\epsilon$$
which shows that the RHS of (\ref{esti}) tends to zero when $n\to \infty$ under our assumption (\ref{assumption1}).

We now prove statement {\rm \bf B}. We follow the argument used in the proof of Lemma 3.4 from \cite{CF14}. 
Let $F_0(S)$ denote the set of vertices of $S$ and let $J:F_0(S)\to [n]$ be an embedding. Consider the indicator random variable 
$X_J:\Omega^{r}_n\to \{0,1\}$ taking the value $1$ on a subcomplex $Y\in \Omega_n^r$ iff $J$ extends to a simplicial embedding $J:S\to Y$. 
Then $X=\sum_{J}X_J$ counts the number of copies of $S$ in $Y$. We wish to show that $X>0$ a.a.s.

We have
$$\mathbb{E}(X)= (1+o(1))\cdot n^{f_0(S)}\cdot \prod_{i=0}^r p_i^{f_i(S)}\, \, \to\infty$$ 
By Chebyshev inequality,
$$\mathbb{P}(X=0)\, \leq\,  \frac{\var(X)}{\mathbb{E}(X)^2}.$$
Our goal is to show that $\frac{\var(X)}{\mathbb{E}(X)^2}\to 0$.
One has 
$$\var(X)=\sum_{J,J'} \cov(X_J,X_{J'})=\sum_{J,J'} (\mathbb{E}(X_JX_{J'})-\mathbb{E}(X_{J})\mathbb{E}(X_{J'})).$$
The product $X_J\cdot X_{J'}$ is the indicator random variable for the containment of the simplicial complex $J(S)\cup J'(S)$.
Hence we have
$$\mathbb{E}(X_{J}X_{J'})=\prod_{i=0}^r p_i^{2f_i(S)-f_i(T')}$$
where $T'=J(S)\cap J'(S)$.

Denote by $T$ the subcomplex $T=J^{-1}(T')\subset S$. For a fixed subcomplex $T\subset S$ the number of pairs of embeddings $J,J':S\to \Delta_n$ such that $J^{-1}(J(S)\cap J'(S))=T$ is bounded above by 
$$C_Tn^{2f_0(S)-f_0(T)}$$
where $C_T$ denotes the number of isomorphic copies of $T$ in $S$.
We obtain
\begin{align*}
\var(X)&\leq \sum_{T\subset S} C_Tn^{2f_0(S)-f_0(T)}\cdot \prod_{i=0}^{r}p_i^{2f_i(S)-f_i(T)}\\
&\leq C\cdot \mathbb{E}^2(X)\cdot \sum_{T\subset S}C_T n^{-f_0(T)}\cdot \left(\prod_{i=0}^rp_i^{-f_i(T)}-1\right)\\
&=C'\cdot \mathbb{E}^2(X)\cdot \sum_{T\subset S, T\neq\emptyset} n^{-f_0(T)}\prod_{i=0}^rp_i^{-f_i(T)}\\
&=C'\cdot \mathbb{E}^2(X)\cdot \sum_{T\subset S, T\neq\emptyset} \left[n^{f_0(T)}\prod_{i=0}^rp_i^{f_i(T)}\right]^{-1}.
\end{align*}
It follows that
$$\frac{\var(X)}{\mathbb{E}(X)^2}\to 0$$ since we assume that $n^{f_0(T)}\prod_{i=0}^rp_i^{f_i(T)}\to\infty$ for all non-empty subcomplexes $T\subset S$.
\end{proof}

\begin{example}\label{circle} {\rm Let $C$ be a simplicial loop of length $m\ge 3$. Then for any non-empty subcomplex $T\subset C$
on a has $f_1(T) \le f_0(T)$ and therefore 
$$n^{f_0(T)}\prod_{i=0}^r p_i^{f_i(T)} \ge \left[np_0p_1\right]^{f_0(T)}.$$
From Theorem \ref{thm1} we obtain that a random complex $Y\in \Omega_n^r$ contains $C$ as a subcomplex assuming that 
$
np_0p_1\to \infty.
$
}
\end{example}

Next we shall consider a relative version of Theorem \ref{thm1}. 
Suppose that $S_1\supset S_2$ are two simplicial complexes; we want to describe conditions when a random simplicial complex 
$Y\in \Omega_n^r$ admits an embedding $S_2\to Y$ which cannot be extended to an embedding $S_1\to Y$. 

By the definition, an embedding of $S_1$ into $Y$ is an injective simplicial map $S_1\to Y$.

We denote 
$f_i(S_1, S_2) = f_i(S_1) - f_i(S_2)$ where $i=0, \dots$. 

\begin{theorem}\label{thm2}
Assume that for any nonempty subcomplex $T\subset S_2$ one has 
\begin{eqnarray}\label{inf}n^{f_0(T)}\cdot \prod_{i=0}^r p_i^{f_i(T)} \to \infty\end{eqnarray}
and besides, 
\begin{eqnarray}\label{zero}n^{f_0(S_1, S_2)}\cdot \prod_{i=0}^r p_i^{f_i(S_1, S_2)} \to 0.\end{eqnarray}
Then the number of embeddings of $S_1$ into a random simplicial complex $Y\in \Omega_n^r$ is smaller than the number of embeddings of $S_2$ 
into $Y$, a.a.s. In particular, under the assumptions (\ref{inf}), (\ref{zero}) there exists an embedding of $S_2 \to Y$ which does not extend to an embedding $S_1\to Y$, a.a.s.
\end{theorem}

\begin{proof}[Proof of Theorem 2]
For $i=1,2$, 
let $X_i: \Omega_n^r\to \mathbb{Z}$ be the random variable that counts the number of embeddings of $S_i$ into the random complex $Y\in\Omega_n^r$.

Our goal is to show that $X_1<X_2$, a.a.s.

 We have
$$\frac{\E(X_1)}{\E(X_2)} \leq C \frac{ n^{f_0(S_1)}\prod_{i=0}^{r}p_i^{f_i(S_1)}}{n^{f_0(S_2)}\prod_{i=0}^{r}p_i^{f_i(S_2)}}
 = C\cdot n^{f_0(S_1,S_2)}\prod_{i=0}^rp_i^{f_i(S_1,S_2)}\to 0$$
tends to zero because of our assumption (\ref{zero}). 

There are two sequences $x_1=x_1(n)$ and $x_2=x_2(n) >0$ such that
$$x_1+x_2=\E(X_2)-\E(X_1)$$ and $\E(X_1)/x_1\to 0$ while $\E(X_2)/x_2$ is bounded.
The sequences $$x_1=\sqrt{\E(X_1)\E(X_2)} \quad \mbox{and}\quad  x_2= \E(X_2) - \E(X_1) - x_1$$ satisfy the desired properties; 
indeed, we have
$$\frac{\E(X_1)}{x_1} = \sqrt{\frac{\E(X_1)}{\E(X_2)}} \to 0$$
and
$$\frac{\E(X_2)}{x_2} = \frac{1}{1-\frac{\E(X_1)}{\E(X_2)}- \sqrt{\frac{\E(X_1)}{\E(X_2)}}}\to 1$$
is bounded.

For every $n$ we have the inequality
\begin{eqnarray}\label{ineq3}
\PP_{r, n}(X_1<X_2) \ge 1- \PP_{r, n}(X_1\geq \E(X_1) +x_1) - \PP_{r, n}(X_2< \E(X_2) -x_2)\end{eqnarray}
since $$\E(X_1)+x_1 =\E(X_2)-x_2$$ and 
if $X_1\geq X_2$ then either $X_1\geq \E(X_1)+x_1$ or $X_2<\E(X_2)-x_2$.

By the Markov inequality we have
$$\PP_{r, n}(X_1\geq\E(X_1)+x_1)\leq \frac{\E(X_1)}{\E(X_1)+x_1}= \frac{\E(X_1)/x_1}{1+ \E(X_1)/x_1}\to 0.$$
By Chebyschev's inequality
$$\PP_{r, n}(X_2< \E(X_2)-x_2) < \frac{{\rm {Var(X_2)}}}{x_2^2}\to 0.$$
Indeed, we have already seen in the proof of Theorem \ref{thm1} that the assumption (\ref{inf}) implies that
$\frac{{\rm {Var(X_2)}}}{\E(X_2)^2}\to 0$ and hence $\frac{{\rm {Var(X_2)}}}{x_2^2}$ tends to zero. Incorporating this information into (\ref{ineq3}) we obtain
$\PP_{r, n}(X_1<X_2)\to 1$. 
\end{proof}

\begin{theorem}\label{thm3} Let 
\begin{eqnarray}S_j \supset S, \quad j=1, \dots, N
\end{eqnarray}
be a finite family of finite simplicial complexes of dimension $\le r$ containing a given simplicial complex $S$ and satisfying the following conditions:

(a) for every nonempty subcomplex $T\subset S$ one has 
\begin{eqnarray}\label{inf1}n^{f_0(T)}\prod_{i=0}^r p_i^{f_i(T)} \to \infty;\end{eqnarray}

(b) 
\begin{eqnarray}\label{zero1}n^{f_0(S_j, S)}\prod_{i=0}^r p_i^{f_i(S_j, S)} \to 0.
\end{eqnarray}
Then with probability tending to one, a random simplicial complex $Y\in \Omega_n^r$ admits an embedding of $S$ which 
does not extend to an embedding $S_j\to Y$, for every $j=1, \dots, N$.
\end{theorem}

The result follows automatically from Theorem \ref{thm2}. Examples when Theorem \ref{thm3} can be applied to produce interesting results appear later, see \S \ref{nontriviality} and \S \ref{torsion2}. 

\section{Uniform Hyperbolicity}

In this section we state a theorem about uniform hyperbolically of random simplicial complexes. This theorem plays a crucial role later in this paper.

First we recall the relevant definitions. Let $X$ be a finite simplicial complex. For a simplicial loop in the 1-skeleton $\gamma: S^1\to X^{(1)}\subset X$, we denote by $|\gamma|$ {\it the length} of 
$\gamma$, i.e. the number of edges composing $\gamma$. If $\gamma$ is null-homotopic,
$\gamma \sim 1$, we denote by $A_X(\gamma)$ {\it the area} of $\gamma$, i.e. the minimal number of $2$-simplices in any simplicial filling $V$ for $\gamma$.
{\it A simplicial filling} (or a simplicial Van Kampen diagram) for a loop $\gamma$ is defined as a pair of simplicial maps $S^1\stackrel{i}\to V\stackrel{b}\to X$
(where $V$ is a contractible 2-dimensional complex) such that $\gamma=b\circ i$ and the mapping cylinder of $i$ is a disc with boundary $S^1\times 0$, see \cite{BHK}.

Next one defines {\it the isoperimetric constant} of $X$:
$$I(X)= \inf \left\{\frac{|\gamma|}{A_X(\gamma)}; \quad \gamma: S^1\to X^{(1)}, \gamma\sim 1 \quad \mbox{in $X$}\right\}\, \in \, \R.$$
Clearly $I(X)=I(X^{(2)})$, i.e. the isoperimetric constant $I(X)$ depends only on the 2-skeleton $X^{(2)}$.
It is well-known that the positivity $I(X)>0$ depends only on the fundamental group $\pi_1(X)$; in fact, one of the many equivalent definitions of hyperbolicity of discrete groups 
in the sense of M. Gromov \cite{Gromov87} states that $\pi_1(X)$ is hyperbolic iff $I(X)>0$. Clearly, the precise value of $I(X)$ depends on the simplicial structure of $X$. Knowing $I(X)$ (or a lower bound for it) is extremely useful as we shall demonstrate later in this paper. 

Below is the main result of this section:

\begin{theorem}\label{hyp}
Consider a random simplicial complex $Y\in \Omega_n^r$ with respect to the probability measure 
$\mathbb{P}_{r,\mathfrak{p}}$, where $\mathfrak{p}=(p_0,\ldots,p_r)$, 
$r\geq 2$. Assume that 
\begin{eqnarray}\label{numberofvertices}
np_0\to \infty
\end{eqnarray} and for some $\epsilon>0$,
\begin{eqnarray}\label{epsilon}
(np_0)^{1+\epsilon} p_1^3 p_2^2\, \to \, 0.\end{eqnarray}
Then there exists a constant $c_\epsilon>0$ (depending only on $\epsilon$) such that a random complex $Y\in \Omega_n^r$ has the following property with probability tending to one: any subcomplex $Y'\subset Y$ satisfies \begin{eqnarray}I(Y')\geq c_\epsilon.\end{eqnarray}
In particular, any subcomplex $Y'\subset Y$ has a Gromov hyperbolic fundamental group, a.a.s..
\end{theorem}

Thus, under the assumption (\ref{epsilon}), all random complexes and their subcomplexes have isoperimetric constants bounded below by $c_{\epsilon}>0$ with probability tending to $1$ as $n\to \infty$. 

Assumption (\ref{numberofvertices}) guaranties that the number of vertices of $Y$ tends to infinity, see \cite{CF14}, Lemma 2.5. 

Using Theorem \ref{hyp} we shall establish the following Corollary:

\begin{corollary} {\rm [see Theorem \ref{nontrivial}]}
If additionally to the hypothesis of Theorem \ref{hyp} one has 
\begin{eqnarray}\label{meaning}np_0p_1\to\infty\end{eqnarray}
as $n\to\infty$ 
then a random complex $Y\in \Omega_n^r$ has non-trivial fundamental group, a.a.s.
\end{corollary}

To illustrate the importance of the assumption (\ref{meaning}) note that 
{\it the \lq\lq alternative\rq\rq}\, assumption
\begin{eqnarray}\label{forest}
np_0p_1\to 0
\end{eqnarray} implies (as is easy to show) that
a random complex $Y\in \Omega_n^r$ is at most one dimensional and has no cycles, i.e. it is a forest, a.a.s. 
In particular, the fundamental group of $Y$ is trivial with any base point, a.a.s. 
Note also that (\ref{forest}) implies that $Y$ is disconnected, see Lemma 5.4 in \cite{CF15}. 

\begin{corollary}\label{cor22}
Assume that $p_i=n^{-\alpha_i}$, where $\alpha_i\ge 0$ are constants, $i=0, \dots, r$.
Then a random complex $Y\in \Omega_n^r$ has hyperbolic and nontrivial fundamental group for 
$$\begin{array}{ll}
\alpha_0+3\alpha_1+2\alpha_2>1,\\
\alpha_0+\alpha_1<1.
\end{array}
$$
If either $\alpha_0+3\alpha_1+2\alpha_2<1$ or $\alpha_0+\alpha_1>1$ then the fundamental group $\pi_1(Y)$ is trivial, a.a.s.
\end{corollary}

The last part of Corollary \ref{cor22} is proven in an earlier paper \cite{CF15a}.

The Figure \ref{fig:nontrivial} depicts the domain where the fundamental group of the random complex is nontrivial. Here we assume that $\alpha_0=\alpha_3=\alpha_4 =\dots=0$. 
%
%
%

Next we recall {\it the local-to-global principle} of Gromov which plays a crucial role in the proof of Theorem \ref{hyp}: 

\begin{theorem}\label{localtoglobal} Let $X$ be a finite 2-complex and let $C>0$ be a constant such that any pure subcomplex $S\subset X$ having at most
$(44)^3\cdot C^{-2}$ two-dimensional simplexes satisfies $I(S)\ge C$. Then $I(X)\ge C\cdot 44^{-1}$.
\end{theorem}

Now we give the following definition.

\begin{definition}\label{eps} Let  $\epsilon>0$ be a positive number. We shall say that a finite 2-dimensional simplicial complex $S$ is {\it $\epsilon$-admissible} if 
the following system of linear inequalities 
\begin{eqnarray}\label{ineq4}
\left\{
\begin{array}{l}
3\alpha_1+ 2\alpha_2 > 1+\epsilon,\\ 
\alpha_1f_1(T)+ \alpha_2f_2(T) < f_0(T), \quad T\subset S,
\end{array}
\right.
\end{eqnarray}
admits a solution with $\alpha_1\ge 0$, $\alpha_2\ge 0$,
where $T\subset S$ runs over all non-empty subcomplexes. In other words, a 2-complex $S$ is $\epsilon$-admissible if there exist non-negative real 
numbers $\alpha_1$ and 
$\alpha_2$ such that $3\alpha_1 +2\alpha_2>1+\epsilon$ and for any non-empty subcomplex $T\subset S$ one has
$\alpha_1f_1(T)+ \alpha_2f_2(T) < f_0(T)$. 

We shall say that a simplicial complex $S$ is admissible if it is $\epsilon$-admissible for some $\epsilon>0$. 
\end{definition}

The property of being $\epsilon$-admissible is a combinatorial property of a 2-complex $S$ which amounts to certain restrictions on the numbers of vertices, edges and faces for all subcomplexes of $S$. We may mention two special cases: 

Case A: A complex $S$ is $\epsilon$-admissible if for any subcomplex $T\subset S$ one has $$\mu_2(T) \equiv \frac{f_0(T)}{f_2(T)} > \frac{1}{2}+\frac{\epsilon}{2}.$$
This is equivalent to Definition \ref{eps} under an additional assumption that $\alpha_1=0$. 
Complexes with $\mu_2(S) >1/2$ were studies in \S 2 of \cite{CF1}. 

Case B: A complex $S$ is $\epsilon$-admissible if for any subcomplex $T\subset S$ one has $$\mu_1(T) \equiv \frac{f_0(T)}{f_1(T)} > \frac{1}{3}+\frac{\epsilon}{3}.$$
This is equivalent to Definition \ref{eps} under an additional assumption that $\alpha_2=0$. Complexes with $\mu_1(S) >1/3$ were studies in \S 5 of 
\cite{CF2}. 


\begin{example} \rm
Let $X$ be a simplicial complex $X$ homeomorphic to the torus $T^2$. 
Then using the Euler characteristic relation we obtain $f_1(X) =2f_0(X)$ and $f_2(X) = 2f_0(X)$. The inequality 
$\alpha_1f_1(X) + \alpha_2f_2(X) < f_0(X)$ (see the second line in (\ref{ineq4})) is equivalent to $3\alpha_1+2\alpha_2 <1$ which contradicts the first line of (\ref{ineq4}). 
Hence the is no $\epsilon>0$ such that $X$ is $\epsilon$-admissible. 
%
\end{example}

\begin{remark}\rm A 2-complex $S$ is said to be {\it balanced} if for any subcomplex $T\subset S$ one has $\mu_i(T)\ge \mu_i(S)$ for $i=1, 2$; see \S 4 from \cite{CF14}. Recall that 
$\mu_i(T)$ denotes the ratio $f_0(T)/f_i(T)$. A balanced 2-complex $S$ is $\epsilon$-admissible iff 
\begin{eqnarray}\label{ineq5}\alpha_1f_1(S) +\alpha_2f_2(S) <f_0(S)\end{eqnarray}
for some $\alpha_1\ge 0$ and $\alpha_2\ge 0$ satisfying $3\alpha_1+2\alpha_2>1+\epsilon$. Note that any triangulation of a closed surface with non-negative Euler characteristic is balanced, see Theorem 4.4 from \cite{CF14}. 
\end{remark}

\begin{example}\label{exsurface} \rm Let $S$ be a connected 2-complex homeomorphic to a closed surface with positive Euler characteristic, 
$\chi(S)>0$. Then $S$ is $\epsilon$-admissible assuming that 
\begin{eqnarray}\label{ineq6}
f_0(S) \le \chi(S) \cdot \frac{1+\epsilon}{\epsilon}.
\end{eqnarray}
Indeed, $S$ is balanced (see above) and the well-known Euler type relations imply 
$$f_1(S) = 3(f_0(S)-\chi(S)), \quad f_2(S) = 2(f_0(S) -\chi(S)).$$
The inequality (\ref{ineq5}) can be rewritten in this case as
$$(3\alpha_1+2\alpha_2)\cdot (f_0(S)-\chi(S))<f_0(S)$$
and hence the system (\ref{ineq4}) is equivalent to 
$$1+\epsilon < 3\alpha_1+2\alpha_2 < \frac{f_0(S)}{f_0(S)-\chi(S)}.$$
We see that the existence of $\alpha_1, \alpha_2$ follows from the inequality
$$1+\epsilon < \frac{f_0(S)}{f_0(S)-\chi(S)}$$
which is equivalent to (\ref{ineq6}). We obtain that any triangulated sphere or projective plane with sufficiently {\it \lq\lq small\rq\rq}\,   number of vertices (as prescribed by (\ref{ineq6})) is $\epsilon$-admissible. 

Besides, all triangulations of the 2-sphere and of  the 
real projective plane are admissible. 
\end{example}

\begin{example} \rm 
Any graph is $\epsilon$-admissible. Indeed, since $f_2(T)=0$, 
in the Definition \ref{eps} one may take $\alpha_1$ very small and $\alpha_2$ very large.
\end{example}

The importance of the notion of $\epsilon$-admissibility stems from the following Lemma:

\begin{lemma}\label{small}
Assume that $np_0\to \infty$ and for some $\epsilon>0$, 
\begin{eqnarray}\label{ineq33}
(np_0)^{1+\epsilon}p_1^3p_2^2 \to 0.
\end{eqnarray}
For a fixed constant $C>0$, a random simplicial complex $Y\in \Omega_n^r$ with probability tending to one has the following property: 
any simplicial pure 2-dimensional subcomplex of $Y$ with at most $C$ 2-simplexes is $\epsilon$-admissible. 
In other words, under the condition (\ref{ineq33}), consider the finite set ${\mathcal F}_{C,\epsilon}\, =\, \{X\}$ of isomorphism classes of pure 2-dimensional simplicial complexes $X$ satisfying
$f_2(X)\le C$ which are not $\epsilon$-admissible. Then with probability tending to 1 a random complex $Y\in \Omega_n^r$ with respect to the multiparameter measure $\p=(p_0, \dots, p_r)$ contains none of the complexes $X\in {\mathcal F}_{C, \epsilon}$ as a simplicial subcomplex. 
\end{lemma}

\begin{proof}
We may write
$$p_i=n^{-\alpha_i},$$
where in general $\alpha_i=\alpha_i(n)$ is a function of $n$. 
By our assumption $np_0\to \infty$, we have $\alpha_0(n)<1$ for all large $n$, i.e. for all $n$ except finitely many. 
Our assumption (\ref{ineq33}) implies that 
$$3\alpha_1+2\alpha_2-(1-\alpha_0)(1+\epsilon) = \frac{\omega}{\log n},$$
where $\omega \to \infty$. This can be rewritten as 
$$3\alpha'_1+2\alpha'_2= 1+\epsilon + x,$$
where
$$\alpha_i'=\frac{\alpha_i}{1-\alpha_0}, \quad x= \frac{\omega}{(1-\alpha_0)\cdot \log n}.$$
Note that $x\le 5\cdot \max\{\alpha'_1, \alpha'_2\}$. For $i=1, 2$ define 
$$\beta_i(n) = \left\{
\begin{array}{ll} 
\alpha'_i(n) - x/10, & \mbox{if}\quad \alpha'_i(n)=\max\{\alpha_1'(n), \alpha_2'(n)\},\\ \\
\alpha'_i(n), & \mbox{otherwise}.
\end{array}
\right.
$$
Then $$\beta_i(n)\ge 0, \quad \mbox{and}\quad 3\beta_1(n)+2\beta_2(n)>1+\epsilon$$ 
for any $n$.

Let $S$ be a simplicial complex with $f_2(S)\le C$ which is not $\epsilon$-admissible. As follows from Definition \ref{eps}, for any $n$ there is a subcomplex $T_n\subset S$ such that 
$$\beta_1f_1(T_n)+\beta_2f_2(T_n) \ge f_0(T_n),$$ which implies that 
$$\alpha'_1f_1(T_n) +\alpha'_2f_2(T_n) -f_0(T_n)\ge \frac{x}{10}.$$
Therefore 
\begin{eqnarray*}
(np_0)^{f_0(T_n)}\cdot \prod_{i=1}^r p_i^{f_i(T_n)}  &=&  \left[n^{1-\alpha_0}\right]^{f_0(T_n)-\sum_{i=1}^2\alpha'_i f_i(T_n)}\\
&\le & \left[n^{1-\alpha_0}\right]^{-x/{10}} = e^{-\omega/10}.
\end{eqnarray*}

Now we apply Theorem \ref{thm1} to conclude that the probability that $S$ is embeddable into $Y$ tends to zero as $n\to \infty$. 
\end{proof}

A crucial role in the proof of Theorem \ref{hyp} plays the following theorem stating that all $\epsilon$-admissible 2-dimensional complexes admit a universal lower bound on the value of their isoperimetric constant:

\begin{theorem}\label{uniform}
Given $\epsilon >0$ there exists a constant $C_\epsilon>0$ such that for any $\epsilon$-admissible finite simplicial pure 2-complex $X$
one has $I(X) \ge C_\epsilon$.
\end{theorem}

\begin{proof}[\bf Proof of Theorem \ref{hyp} using Theorem \ref{localtoglobal} and Theorem  \ref{uniform}.]
Let $C_\epsilon>0$ be the constant given by Theorem \ref{uniform}.
Consider the set $\mathcal S$ of all isomorphism types of finite 
pure 2-complexes having at most $C=44^3\cdot C_\epsilon^{-2}$ 
two-dimensional simplexes. 
Clearly, the set $\mathcal S$ is finite.
Let $\mathcal{S}'\subset\mathcal S$ denote the subset of complexes in $\mathcal S$ which are not $\epsilon$-admissible.
Under the assumptions of Theorem \ref{hyp} a random complex $Y\in \Omega_n^r$ contains a complex from $\mathcal S'$  as a 
subcomplex with probability tending to zero as $n\to \infty$ as follows from Lemma \ref{small}.

All remaining complexes lying in $\mathcal S''=\mathcal S-\mathcal S'$ are $\epsilon$-admissible. Theorem \ref{uniform} states that 
any complex $S\in\mathcal S''$ satisfies $I(S)\geq C_\epsilon$.
Now applying Theorem \ref{localtoglobal} we obtain that any subcomplex $Y'\subset Y$ satisfies 
$I(Y')\ge
C_\epsilon\cdot 44^{-1}=c_\epsilon$ with probability tending to $1$ as $n\to \infty$. 

\end{proof}



Theorem \ref{uniform} will be proven in \S \ref{app}.

\section{Topology of admissible 2-complexes}
\label{admis}

In this section we examine the topology of 2-complexes which are admissible in the sense of Definition \ref{eps}. One of the central results proven here is Theorem \ref{wedge} describing homotopy types of admissible 2-complexes. 
In \S \ref{app} we shall continue the study of $\epsilon$-admissible complexes and present a proof of Theorem \ref{uniform}. 

Recall the definitions of the density invariants:
\begin{eqnarray}\label{dense}
\mu_i(S)= \frac{f_0(S)}{f_i(S)}, \quad i=1, 2. 
\end{eqnarray}
We shall use the formulae
\begin{eqnarray}\label{eq7}
\mu_1(S) = \frac{1}{3} + \frac{3\chi(S) +L(S)}{3f_1(S)}, \quad \mu_2(S) = \frac{1}{2} + \frac{2\chi(S) +L(S)}{2f_2(S)},
\end{eqnarray}
see formula (8) in \cite{CF2} and formula (2) in \cite{CF1}. Here
\begin{eqnarray}
L(S) = \sum_e \left(2-\deg e\right),
\end{eqnarray}
the sum is taken over the edges $e$ of $S$ and for an edge $e$ the symbol $\deg e$ (the degree of $e$) denotes the number of 2-simplexes containing $e$. 

\begin{lemma}\label{mu1 or mu2}
Let $S$ be an $\epsilon$-admissible 2-complex. 
Then for any subcomplex $T\subset S$ either 
$$\mu_1(T)>\frac{1+\epsilon}{3},\quad \textrm{or}\quad \mu_2(T)>\frac{1+\epsilon}{2}.$$
In particular, if $S$ is admissible then for any subcomplex $T\subset S$ one has either
\begin{eqnarray*}
3\chi(T)+L(T)>0,\quad \textrm{or}\quad 2\chi(T)+L(T)>0.
\end{eqnarray*}
 Moreover, if $S$ is admissible then for any subcomplex $T\subset S$ with $L(T)\leq 0$ one has $$\mu_1(T)>1/3.$$
\end{lemma}
\begin{proof} Suppose that for some $T\subset S$ one has $$\mu_1(T) \le (1+\epsilon)/3, \quad \mbox{and} \quad \mu_2(T) \le (1+\epsilon)/2.$$
Then for any $\alpha_1, \alpha_2\ge 0$ satisfying $3\alpha_1+2\alpha_2\geq 1+\epsilon$ one has 
$$\frac{\alpha_1}{\mu_1(T)}+\frac{\alpha_2}{\mu_2(T)} >1$$ 
which is equivalent to 
$$\alpha_1f_1(T) + \alpha_2f_2(T) >f_0(T),$$
implying that $S$ is not $\epsilon$-admissible. 

The other statements follow from formulae (\ref{eq7}). 
\end{proof}

\begin{corollary} Let $S$ be a simplicial 2-complex homeomorphic to a closed surface. If $\chi(S)\le 0$ then $S$ is not admissible. 
\end{corollary}
\begin{proof}
Applying the previous Lemma \ref{mu1 or mu2} with $T=S$ and observing that $L(S)=0$ we have 
$$\mu_1(S) = \frac{1}{3} + \frac{\chi(S)}{f_1(S)}\, \le \, \frac{1}{3}$$
which shows that $S$ is not admissible due to Lemma \ref{mu1 or mu2}. 
\end{proof}

In Example \ref{exsurface} we showed that any closed surface $S$ with $\chi(S)>0$ is admissible. 

Recall that a 2-complex $S$ is called {\it closed} if every edge $e$ of $S$ is contained in at least two 2-simplexes. Note that for a closed complex $S$ one has $L(S)\le 0$. 
A 2-complex $S$ is said to be {\it pure} if each edge $e$ of $S$ is contained in at least one 2-simplex. 

\begin{corollary}\label{b2=0}
Any closed strongly connected 2-dimensional  admissible simplicial  complex $S$ with $b_2(S)=0$ is either a triangulation of the real projective plane $P^2$ or the quotient $Q^2$ of a triangulation of $P^2$ obtained by identifying two adjacent edges.
\end{corollary}
\begin{proof}
By Lemma \ref{mu1 or mu2} we have that $\mu_1(S)>1/3.$
Lemma 5.1 
of \cite{CF2} implies that $S$ is a triangulated projective plane $P^2$ or the quotient of a triangulation of $P^2$ obtained by identifying two adjacent edges.
\end{proof}

\begin{definition}
A finite simplicial 2-complex $Z$ is said to be a minimal cycle if $b_2(Z)=1$ and for any proper subcomplex $Z'\subsetneq Z$ one has $b_2(Z')=0$.
\end{definition}

\begin{definition}
A minimal cycle $Z$ is said to be of type A if it does not contain closed proper subcomplexes. 
Otherwise $Z$ is said to be of type B.
\end{definition}

\begin{example}\label{ex1} \rm 
Examples of minimal cycles are:

\begin{enumerate}

\item  A triangulation of the sphere $S^2$ is an admissible minimal cycle of type A, see Example \ref{exsurface}; 

\item One also obtains an admissible minimal cycle of type A by 
starting from a triangulation of $S^2$ and identifying two vertices. However, if one identifies more than two vertices the obtained minimal cycle is not admissible. 

\item An admissible minimal cycle of type A is obtained from a triangulation of $S^2$ by identifying two adjacent edges. 

\item An example of a minimal cycle of type B is given by the union $Z= P^2\cup D^2$ where $P^2$ is a triangulation of the real projective plane and $D^2$ is a triangulated disc such that $P^2\cap D^2=\partial D^2$ is a non-contractible simple closed loop on $P^2$. This minimal cycle is admissible if $f_1(\partial D^2)\le 5$. 

\item Consider the union $Z=P^2\cup P^2$ of two real projective planes where the intersection $P^2\cap P^2$ is a 
loop non-contractible in each of the projective planes. $Z$ 
is a  minimal cycle of type B which is not admissible
\end{enumerate}
\end{example}

\begin{remark} \label{23} \rm One can easily see that minimal cycles are closed, strongly connected simplicial complexes. Hence, for any minimal cycle $Z$ one has $L(Z)\le 0$; besides, $\chi(Z)\le 2$. 
Using Lemma \ref{mu1 or mu2} we find that every admissible minimal cycle $Z$ must satisfy $3\chi(Z)+L(Z)>0$ implying that
\begin{align}
1\leq \chi(Z)\leq 2\quad \textrm{and}\quad -5\leq L(Z)\leq 0.
\end{align}
\end{remark}

\begin{lemma}\label{typeA}
Any admissible minimal cycle $Z$ of type A is homotopy equivalent to either $S^2$ or to $S^2\vee S^1$. Moreover, for every 2-simplex $\sigma\subset Z$ the boundary $\partial \sigma$ is null-homotopic in $Z-\sigma$.
\end{lemma}

\begin{proof} 
Suppose $Z$ is an admissible minimal cycle of type A. Since $Z$ is closed, one has $L(Z)\leq 0$. Using the second part of 
Lemma \ref{mu1 or mu2}, we obtain
$\mu_1(Z)>1/3$. Now we may apply Lemma 5.6 from \cite{CF2}. Note that the statement of Lemma 5.6 from \cite{CF2} 
requires that $\mu_1(T)>1/3$ for any subcomplex $T\subset Z$; however the proof presented in \cite{CF2} uses only 
the assumption $\mu_1(Z)>1/3$. 
%
\end{proof}

Next we establish the following simple fact about minimal cycles of type B which strengthens Remark \ref{23}. 

\begin{lemma}\label{proper}
Every admissible minimal cycle $Z$ of type $B$ satisfies $\chi(Z)=2$ and $-5\le L(Z)\leq -3$.
\end{lemma}
\begin{proof}
Let $Z'$ be a proper closed subcomplex of $Z$. 
If the graph $\Gamma= Z'\cap (\overline{Z-Z'})$ has no cycles, then
$$b_2(Z) = b_2(Z') + b_2(\overline{Z-Z'})$$
 and either $b_2(Z')=1$ or $b_2(\overline{Z-Z'})=1$; either of these possibilities contradicts the minimality of $Z$. 
 Hence the graph $\Gamma$ must contain a cycle and in particular, the number of edges of $\Gamma$ satisfies $f_1(\Gamma)\geq 3$.
Each edge of $Z'$ is incident to at least two faces of $Z'$ and every edge of $\Gamma\subset Z'$  is incident to at least one face of $Z$ which is not in $Z'$. Since $f_1(\Gamma)\geq 3$ and $Z$ is closed it follows that $L(Z)\leq -3$.
Since $Z$ is admissible, by Lemma \ref{mu1 or mu2} we obtain $0<3\chi(Z)+L(Z)\leq 3\chi(Z)-3$. In particular $\chi(Z)>1$. Since $b_0(Z)=b_2(Z)=1$ we obtain $\chi(Z)=2$, as claimed.
\end{proof}

\begin{definition}
Let $Z$ be an admissible minimal cycle of type B. Any closed strongly connected proper subcomplex $Z_0\subset Z$ is called a core of $Z$.
\end{definition}

Clearly, for any core $Z_0\subset Z$ one has $b_2(Z_0)=0$ (by minimality). Applying Corollary \ref{b2=0} we see that any core is homeomorphic either to $P^2$ or to $P^2$ with two adjacent edges identified. 

\begin{lemma}\label{unique} An admissible minimal cycle $Z$ of type B has a unique core $Z_0\subset Z$.
\end{lemma}

\begin{proof} Let $Z$ be an admissible minimal cycle of type B. Let us assume that $Z$ has two distinct cores $Z', Z''\subset Z$. 

Consider the graphs
$$\Gamma'=Z'\cap \overline{Z-Z'}, \quad \Gamma''=Z''\cap \overline{Z-Z''}.$$
By the arguments used in the proof of Lemma \ref{proper} we obtain that $f_1(\Gamma')\ge 3$ and $f_1(\Gamma'')\ge 3$. 
The graphs $\Gamma'$ and $\Gamma''$ cannot be edge-disjoint since otherwise $Z$ 
would have at least 6 edges of degree $\ge 3$; the latter would give $L(Z) \le -6$ contradicting Lemma \ref{proper}.
This implies that $f_1(Z'\cap Z'')>0$ and therefore the union $Z'\cup Z''$ is strongly connected. 

We observe next that the union $Z'\cup Z''$ must coincide with $Z$. Indeed, if $Z'\cup Z''\not=Z$ then by minimality 
$b_2(Z'\cup Z'')=0$ and by Corollary \ref{b2=0} the union $Z'\cup Z''$ is either homeomorphic to $P^2$ or to the quotient $Q^2$ 
of $P^2$ 
where two adjacent edges are identified. But neither $P^2$ nor $Q^2$ admits a triangulation in which it is a union of two distinct closed proper subcomplexes. 
Here we use our assumption $Z'\not= Z''$. 

Now we may show that 
$$\Gamma'= S\cap \overline{Z''-S}, \quad \Gamma'' = S\cap \overline{Z'-S},$$
where $$S=Z'\cap Z''$$ is the intersection. Indeed, since $Z=Z'\cup Z''$ one obtains $Z-Z'=Z''-S$ and 
\begin{eqnarray*}
\Gamma'&=& Z'\cap \overline{Z-Z'} \\ &= &Z'\cap \overline{Z''-S}\\
&=& \left[S\cap \overline{Z''-S}\right]\cup \left[(Z'-S)\cap \overline{Z''-S}\right]\\
&=& S\cap \overline{Z''-S}.
\end{eqnarray*}
On the last step we used the observation $(Z'-S)\cap \overline{Z''-S}\subset (Z'-S)\cap Z''=\emptyset.$ The statement regarding $\Gamma''$ follows similarly.

We know that $\chi(Z'\cup Z'')=2$ and $\chi(Z')=\chi(Z'')=1$; therefore $\chi(S)=0$. 
If $S$ is disconnected then $b_1(Z) = b_1(Z'\cup Z'')\ge 1$ contradicting $\chi(Z)=2$. Hence we obtain 
\begin{eqnarray}\label{equal1} b_0(S)=b_1(S)=1.
\end{eqnarray}

Suppose that the intersection $S$ has no 2-faces. Then $S$ is a connected graph and every edge of $S$ has degree $\ge 4$ in $Z$ 
since it is incident to at least two faces of $Z'$ and two faces of $Z''$. 
However, since $b_1(S)=1$ we have $f_1(S)\ge 3$ implying that $L(Z)\le -2f_1(S)\le -6$; this contradicts $L(Z)\ge -5$, see above. 

Thus we see that $f_2(S)\ge 1$, i.e. the complex $S$ is 2-dimensional.

Note that $b_1(Z')=b_2(Z')=0$ and $b_1(Z'')=b_2(Z'')=0$, see Corollary \ref{b2=0}.
The Mayer-Vietoris exact sequences with rational coefficients for the covers $Z'=S\cup\overline{Z'-S}$ and $Z''=S\cup\overline{Z''-S}$ give the isomorphims
\begin{eqnarray}
\label{iso1}H_1(\Gamma';\mathbb Q)&\cong& H_1(S;\mathbb Q)\oplus H_1(\overline{Z''-S};\mathbb Q),\\
\label{iso2}H_1(\Gamma'';\mathbb Q)&\cong& H_1(S;\mathbb Q) \oplus H_1(\overline{Z'-S};\mathbb Q).
\end{eqnarray}
Since $H_1(S;\mathbb Q)=\mathbb Q$ we obtain from (\ref{iso1}), (\ref{iso2}):
\begin{eqnarray}\label{b11}
b_1(\Gamma'),\,  b_1(\Gamma'')\ge 1.
\end{eqnarray}

In the beginning of the proof we have observed that every edge of $\Gamma'$ and of $\Gamma''$ has degree $\ge 3$ in $Z$ and that $\Gamma'$ and $\Gamma''$ cannot be edge-disjoint. 
Hence $f_1(\Gamma'\cup\Gamma'')\le 5$ and therefore $b_1(\Gamma'\cup\Gamma'')\le 2$ since any graph on at most five edges has at most 2 independent cycles. Let us show that the case $b_1(\Gamma'\cup\Gamma'')= 2$ is impossible. 
If $b_1(\Gamma'\cup\Gamma'')=2$ then $\Gamma'\cup \Gamma''$ is a square with one diagonal, it has 5 edges and each of the edges has degree 3 in $Z$
(since $L(Z)\ge -5$, see above).
Moreover, in this case all other edges of $Z$ have degree 2 (as again follows from $L(Z)\ge -5$). 
Let $v$ be one of the vertices of degree $3$ in the graph $\Gamma'\cup\Gamma''$. Then $v$ is incident to exactly three odd degree edges in $Z$ (see similar argument in \cite{CF2}, proof of Lemma 5.7 on page 15). 
All edges of $Z$ incident to $v$ which do not belong to 
$\Gamma' \cup \Gamma''$ have degree 2 in $Z$. In particular the link ${\rm Lk}_Z(v)$ of $v$ in $Z$ would be a graph with an odd number of odd degree vertices which is impossible. Therefore, we obtain that 
\begin{eqnarray}\label{b12}
b_1(\Gamma'\cup \Gamma'')=1
\end{eqnarray}

From (\ref{b11}) and (\ref{b12}) it follows that 
\begin{eqnarray}
b_1(\Gamma'\cap \Gamma'') \, =\,  b_1(\Gamma')\, =\, b_1(\Gamma'') \, =\, 1
\end{eqnarray} 
i.e. the graphs $\Gamma'$ and $\Gamma''$ possess a common cycle $C$. 
The isomorphisms (\ref{iso1}) and (\ref{iso2}) give
$$b_1(\overline{Z'-S})=b_1(\overline{Z''-S})=0.$$
We obtain that some integral multiple of the cycle $C$ bounds a $\Z$-chains in $\overline{Z'-S}$ and in $\overline{Z''-S}$ 
and the difference of these two chains will be a non-trivial 2-dimensional cycle $c$ lying in $Z-S'\not=Z$ contradicting the minimality of $Z$. 
Here $S'$ denotes the union of interiors of all 2-simplexes of $S$. 
Note that the complexes $\overline{Z'-S}$, $\overline{Z''-S}$ and $S$ have no common 2-simplexes. This completes the proof.

\end{proof}


\begin{lemma} \label{typeB} Any admissible minimal cycle $Z$ of type B is homotopy equivalent to the sphere $S^2$. 
Let $Z_0\subset Z$ denote the core of $Z$. Then for any 2-simplex $\sigma \subset Z_0$ the complement $Z-\int(\sigma)$ is contractible. 
\end{lemma}
\begin{proof} Consider the complex $Z-\int(\sigma)$ where $\sigma$ is a 2-simplex, $\sigma\subset Z_0$ lying in the core. 
Starting from the complex $Z-\int(\sigma)$ and collapsing subsequently faces across the free edges we shall arrive to a connected graph $\Gamma$, as follows from the uniqueness of the core (Lemma \ref{unique}). Since $\chi(Z)=2$ (see Lemma \ref{proper}) we find $\chi(\Gamma)=\chi(Z-\int(\sigma)) =1$. 
Therefore $\Gamma$ is a tree. Hence the complex $Z-\int(\sigma)$ is contractible. This implies that $Z$ is homotopy equivalent to the result of attaching a 2-cell to $Z-\int(\sigma)$, hence $Z\simeq S^2$. \end{proof}

\begin{theorem}\label{wedge}
Any admissible 2-complex $X$ is homotopy equivalent to a wedge of circles, spheres and projective planes.
\end{theorem}

\begin{proof}
We will act by induction on $b_2(X)$.
If $b_2(X) = 0$ and $X$ is admissible then using Corollary \ref{b2=0} we see that each strongly connected component of $X$ is homotopy
equivalent to $P^2$. Hence $X$ is homotopy equivalent to a wedge of circles and projective planes.

Assume now that the statement of the Theorem has been proven for all connected admissible 2-complexes $X$ with $b_2(X) < k$.
Consider an admissible 2-complex $X$ satisfying $b_2(X) = k > 0$.
Find a minimal cycle $Z\subset X$ and observe that the homomorphism 
$H_2(Z; \mathbb{Z}) = \mathbb{Z} \to H_2(X; \mathbb{Z})$ induced by the inclusion is injective.
Let $\sigma\subset Z$ be a 2-simplex; if $Z$ is of type B we shall assume that $\sigma$ lies in the core $Z_0\subset Z$. 
We shall use Lemmas \ref{typeA} or \ref{typeB} depending whether $Z$ is of type A or B.
The complex $X' = X-\int(\sigma)$ satisfies $b_2(X' ) = k-1$ and is admissible and thus by induction $X'$ is homotopy equivalent to a wedge of circles, spheres and projective planes. 
Therefore, $X$ is homotopy equivalent to $X'\vee S^2$ and hence $X$ is homotopy equivalent to a wedge of circles, spheres and projective planes. 
\end{proof}

\begin{corollary}\label{core} The fundamental group of any admissible 2-complex $X$ is the free product of several copies of $\Z$ and $\Z_2$. In particular, $\pi_1(X)$ is hyperbolic. 
\end{corollary}

Corollary \ref{core} is in some sense a weak version of Theorem \ref{uniform} which will be proven later in \S \ref{app}.

\section{Non-triviality of the fundamental groups of random simplicial complexes}\label{nontriviality}

\begin{theorem} \label{nontrivial} Let $Y\in \Omega_n^r$ be a random simplicial complex with respect to the probability measure
$\PP_{r,\p}$ where $\p=(p_0, p_1, \dots, p_r)$. Assume that 
\begin{eqnarray}\label{one}
np_0p_1\to \infty
\end{eqnarray}
and for some $\epsilon>0$,
\begin{eqnarray}\label{two}
(np_0)^{1+\epsilon}p_1^3p_2^2\to 0.
\end{eqnarray}
Then for some choice of the base point $y_0\in Y$  the fundamental group $\pi_1(Y, y_0)$ is nontrivial, a.a.s.
\end{theorem}

\begin{remark}\rm 
If the assumption (\ref{one}) is replaced by the stronger assumption
\begin{eqnarray}
np_0p_1- \log(np_0) \to \infty,
\end{eqnarray}
then $Y\in \Omega_n^r$ is connected, a.a.s. (see Corollary 7.2 from \cite{CF15}). 
However under the assumption (\ref{one}) a random complex might be disconnected, see \S 7 from \cite{CF15}) and thus, the fundamental group 
$\pi_1(Y, y_0)$ might depend on the choice of the base point $y_0\in Y$. 
\end{remark}

\begin{remark}{\rm 
In the special case when $$p_i=n^{-\alpha_i}$$ with $\alpha_i\ge 0$ constant, where $i=0, 1, \dots$ Theorem \ref{nontrivial} states that a random complex 
$Y$ has a nontrivial fundamental group assuming that 
\begin{figure}[t] 
   \centering
   \includegraphics[width=3in]{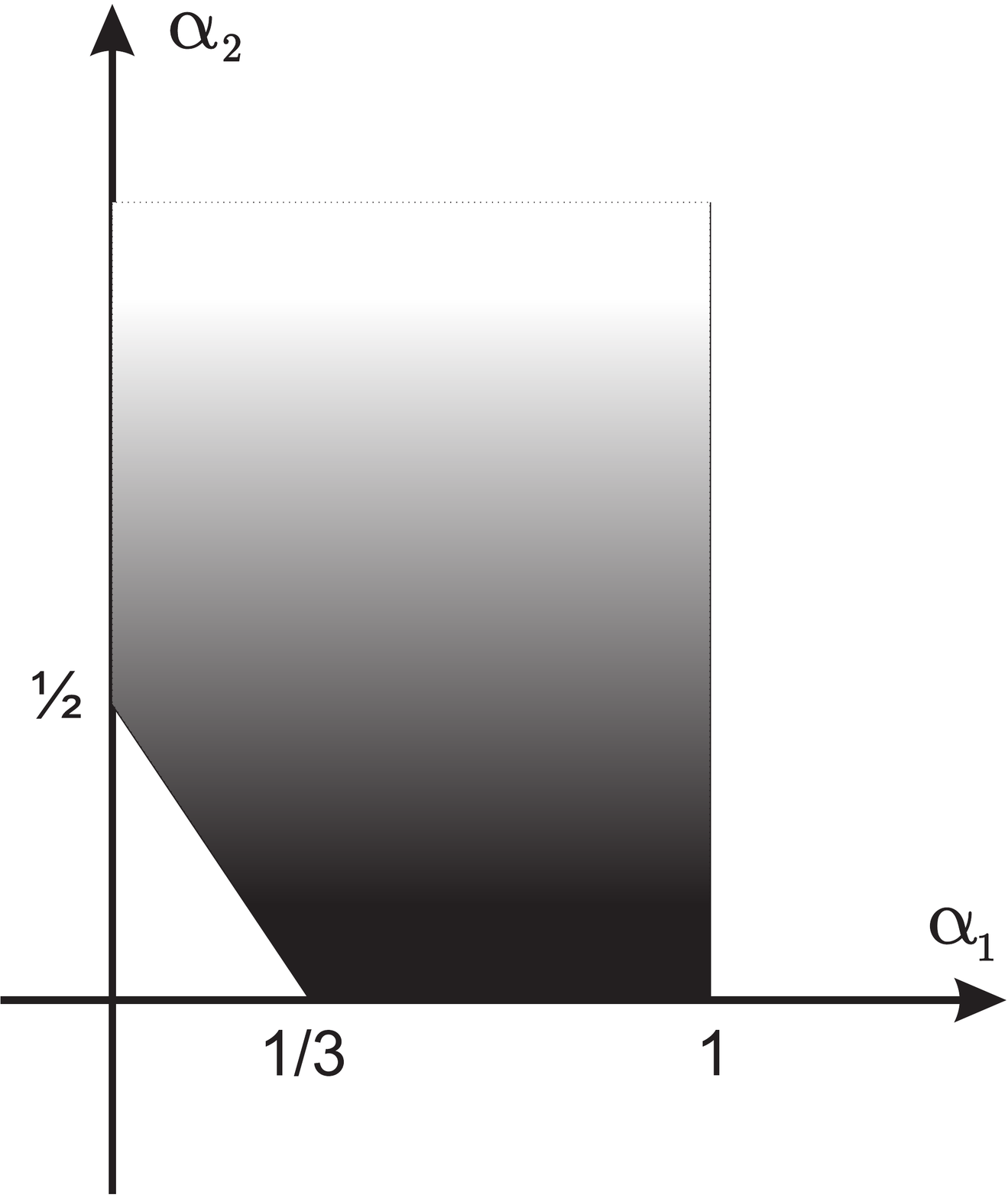} 
   \caption{The region on the plane of $\alpha_1, \alpha_2$ where the fundamental group $\pi_1(Y)$ is nontrivial and hyperbolic in the sense of Gromov.}
   \label{fig:nontrivial}
\end{figure}
\begin{eqnarray*}
&&\alpha_0+\alpha_1 <1, \\
&&\alpha_0 +3\alpha_1+2\alpha_2>1.
\end{eqnarray*}
From paper \cite{CF15a} we know that the inequality $\alpha_0+\alpha_1 <1$ implies connectivity of $Y$. 
}
\end{remark}

\begin{proof}[Proof of Theorem \ref{nontrivial}] 
Let $C$ be the simplicial loop of length 4, i.e. the boundary of the square. 
Using Example \ref{circle}, we note that our assumption (\ref{one}) implies that a random complex $Y\in \Omega_n^r$ contains $C$ as a subcomplex with probability tending to one. 

We want to show that $Y$ contains $C$ 
as {\it \lq\lq an essential subcomplex\rq\rq},\,  i.e. such that the inclusion $C\subset Y$ induces a non-trivial homomorphism $\pi_1(C, y_0)\to \pi_1(Y, y_0)$. This would imply that $\pi_1(Y,y_0)\not=1$ for some choice of the base point.
We shall use Theorem \ref{thm3} to show the existence of essential embeddings $C\to Y$.

Let $c_\epsilon>0$ be the constant given by Theorem \ref{hyp}. 

The arguments of the proof which is presented below may seem formal, and to illustrate them we give in this paragraph a brief vague intuitive description. 
If an inclusion $C\subset Y$ is not essential then $C\to Y$ can be extended to a simplicial map $D\to Y$ where $D$ is a simplicial disc. 
Using the inequality 
$I(Y)\ge c_\epsilon$ (which we may assume to be satisfied due to Theorem \ref{uniform}), we may assume that $D$ has at most $4\cdot c_\epsilon^{-1}$ 
2-simplexes. Of course the map $D\to Y$ does not have to be injective; therefore the image of $D\to Y$ is a subcomplex $S\subset Y$ with at most $4\cdot c_\epsilon^{-1}$ 2-simplexes. There are finitely many isomorphism types of disc triangulations with at most $4\cdot c_\epsilon^{-1}$ 2-simplexes and 
there are finitely many isomorphism types of
their simplicial images $S$;  those $S$'s which are not $\epsilon$-admissible 
appear with probability tending to $0$ (due to Lemma \ref{small}). Below we formalise the properties of complexes $S$ which may appear in this way and use Theorem \ref{thm3} to show that there exists an embedding $S\to Y$ which cannot be extended to an embedding $S\to Y$; this embedding is clearly essential. 

Now we continue with the formal argument. 
%
%
Consider the set $\cal L_\epsilon$ of isomorphism types of pairs
$(S,C)$ where $S$ is a finite 2-complex such that:

(a) the inclusion $C\to S$ induces the trivial homomorphism of the fundamental groups;

(b) $S$ is minimal in the sense that for any proper subcomplex $C\subset S'\subsetneq S$ the inclusion $C\to S'$ induces a non-trivial homomorphism 
$\pi_1(C)\to \pi_1(S')$; 
 
 (c) $f_2(S)\le 4\cdot c_\epsilon^{-1}$;
 
 (d) $S$ is $\epsilon$-admissible. 
 
 Note that for $S\in \mathcal L_\epsilon$ one has $b_2(S)=0$. 
 Indeed, if $b_2(S)\not=0$ then $S$ would contain a minimal cycle $Z\subset S$ and this minimal cycle would be 
 $\epsilon$-admissible. 
 By Lemma \ref{typeA} and Lemma \ref{typeB} the complex $Z$ contains a 2-simplex $\sigma$ such that the boundary 
 $\partial \sigma$ is null-homotopic in $Z-\int(\sigma)$. Hence removing $\sigma$ does not change the fundamental group and we obtain contradiction with the minimality property (b). 
 
 We obtain that $\chi(S)\le 1$ for any $S\in \mathcal L_\epsilon$.
 
 It is easy to see that $S$ cannot have edges of degree zero (i.e. $S$ must be pure) and  any edge of $S$ of degree $1$ must lie in $C$ 
 (as follows from the minimality property (b)).
 Therefore $L(S)\le f_1(C)= 4$. 
 
 We observe that for any $S\in \mathcal L_\epsilon$ one has
 \begin{eqnarray}
 f_1(S,C) \, &\ge&\,  3f_0(S,C)+1,\label{f1}\\
 f_2(S,C) \, &\ge &\, 2f_0(S,C)+2.\label{f2}
 \end{eqnarray}
 Recall that the notation $f_i(S,C)$ stands for $f_i(S)-f_i(C)$, the number of $i$-simplexes lying in $S-C$. 
 For any simplicial complex $X$ of dimension $\le 2$ one has
 \begin{eqnarray*}
 3\chi(X) + L(X) = 3f_0(X) -f_1(X), \\
 2\chi(X) +L(X) = 2f_0(X) -f_2(X). 
 \end{eqnarray*}
 Applying these equalities to $S$ and to $C$ and using 
 $$\chi(S) \le 1, \quad L(S)\le 4, \quad \chi(C)=0, \quad L(C)=8$$ gives (\ref{f1}) and (\ref{f2}). 
 
Now, using (\ref{f1}) and (\ref{f2}) for $S\in \mathcal L_\epsilon$ we obtain
 $$
 n^{f_0(S,C)}\prod_{i=0}^2 p_i^{f_i(S, C)} \le \left[np_0p_1^3p_2^2\right]^{f_0(S,C)}\cdot p_1p_2^2.
 $$
 We claim that our assumptions (\ref{one}) and (\ref{two}) imply that $np_0p_1^3p_2^2\to 0$ and $p_1p_2^2\to 0$. 
 Indeed, (\ref{one}) implies that $np_0\to \infty$ and then (\ref{two}) implies that $np_0p_1^3p_2^2\to 0$. It is easy to see that 
 $p_1p_2^2\to 0$ follows. 
 
 Next we observe that the set $\mathcal L_\epsilon$ is finite and hence we may apply Theorem \ref{thm3}. 
 Since all the assumptions of this theorem are satisfied we obtain that {\it with probability tending to one, for $Y\in \Omega_n^r$ there exists an embedding 
 $C\to Y$ which cannot be extended to an embedding $S\to Y$ for any $S\in \mathcal L_\epsilon$.} 
 
 Consider the set $\Omega'_n\subset \Omega_n^r$ consisting of complexes $Y\in \Omega_n^r$ satisfying the following three conditions:

(1) every subcomplex $Y'\subset Y$ satisfies $I(Y')\ge c_\epsilon$; 

(2) any 2-dimensional pure subcomplex $S\subset Y$ with $f_2(S) \le 4\cdot c_\epsilon^{-1}$ is $\epsilon$-admissible. 

(3) $Y$ contains $C$ as a subcomplex such that no complex $S\in \mathcal L_\epsilon$ can be embedded into $Y$ extending the embedding $C\to Y$. 

By Theorem \ref{uniform}, Lemma \ref{small}, Theorem \ref{thm3} and Example \ref{circle}, the probability that $Y$ belongs to $\Omega'_n$ tends to one as $n\to \infty$. 
We show below that 
for every $Y\in \Omega'_n$ the fundamental group $\pi_1(Y, y_0)$ is non-trivial for some choice of the base point.

 Let us explain that for $Y\in \Omega_n'$
 the embedding 
 $C\to Y$ is essential. 
For any embedding $C\subset Y$ which is not essential there exists a simplicial disc 
$D$ with at most $4\cdot c_\epsilon^{-1}$ 2-simplexes and $\partial D=C$ which can be simplicially mapped into $Y$ extending the embedding $C\to Y$ (this follows from the inequality $I(Y)\ge c_\epsilon^{-1}$ given by Theorem \ref{hyp}). 
The image of this disc is a subcomplex $S\subset Y$ containing $C$ which satisfies the properties (a), (c). Property (d) is satisfied because of Lemma 
\ref{small}. 
If the minimality property (b) is not satisfied, then instead of $S$ we can consider an appropriate smaller subcomplex. 
Hence any non-essential embedding 
$C\to Y$ can be extended to an embedding $S\to Y$ for some $S\in \mathcal L_\epsilon$. 
This completes the proof.
 \end{proof}

\begin{remark}\rm In the proof presented above we showed the existence of essential loops of length 4. The same argument applies for loops of any fixed length $\ge 4$; it also applies to loops of length 3 under an additional assumption $p_2\to 0$. It is obvious that in the case when $p_2=1$ every loop of length 3 is the boundary of a 2-simplex in $Y$ and hence our statement would be false for loops of length 3. 

\end{remark}

\section{The existence of 2-torsion}\label{torsion2}

\begin{theorem} \label{2torsiona} 
Let $Y\in \Omega_n^r$ be a random simplicial complex with respect to the probability measure
$\PP_{r,\p}$ where $\p=(p_0, p_1, \dots, p_r)$. Assume that 
\begin{eqnarray}\label{2}
p_2\to 0,
\end{eqnarray}
\begin{eqnarray}\label{3}
np_0p_1^{5/2}p_2^{5/3}\to \infty,
\end{eqnarray}
and for some $\epsilon>0$,
\begin{eqnarray}\label{4}
(np_0)^{1+\epsilon}p_1^3p_2^2\to 0.
\end{eqnarray}
Then  the fundamental group $\pi_1(Y)$ has elements of order 2, a.a.s.
\end{theorem}

\begin{remark}{\rm 
The assumptions of Theorem \ref{2torsiona} imply that $Y\in \Omega_n^r$  connected, a.a.s. To show this one may use Corollary 7.2 from \cite{CF15a}, 
which requires that 
$\omega\to \infty$ and $\omega p_1-\log(\omega) \to \infty$ where $\omega = np_0$. By (\ref{3}) we have $n^{2/5}p_0^{2/5}p_1=\omega'\to \infty$. 
Hence $\omega p_1 -\log \omega = \omega^{3/5}\cdot \omega' -\log \omega \to \infty$. 
}
\end{remark}

\begin{remark}{\rm 
In the special case when $$p_i=n^{-\alpha_i}$$ with $\alpha_i\ge 0$ constant, where $i=0, 1, \dots$ Theorem \ref{2torsiona} states that the fundamental group of a random complex 
$Y$ has a nontrivial element of order 2 assuming that 
\begin{eqnarray*}
&&\alpha_0+\frac{5}{2} \alpha_1 + \frac{5}{3} \alpha_2<1, \\
&&\alpha_0 +3\alpha_1+2\alpha_2>1,\\
&&\alpha_2>0.
\end{eqnarray*}
This result was proven in \cite{CF1} in the special case when $\alpha_1=0$. See Figure \ref{fig:2torsion}. 

}
\end{remark}
\begin{figure}[t] 
   \centering
   \includegraphics[width=3in]{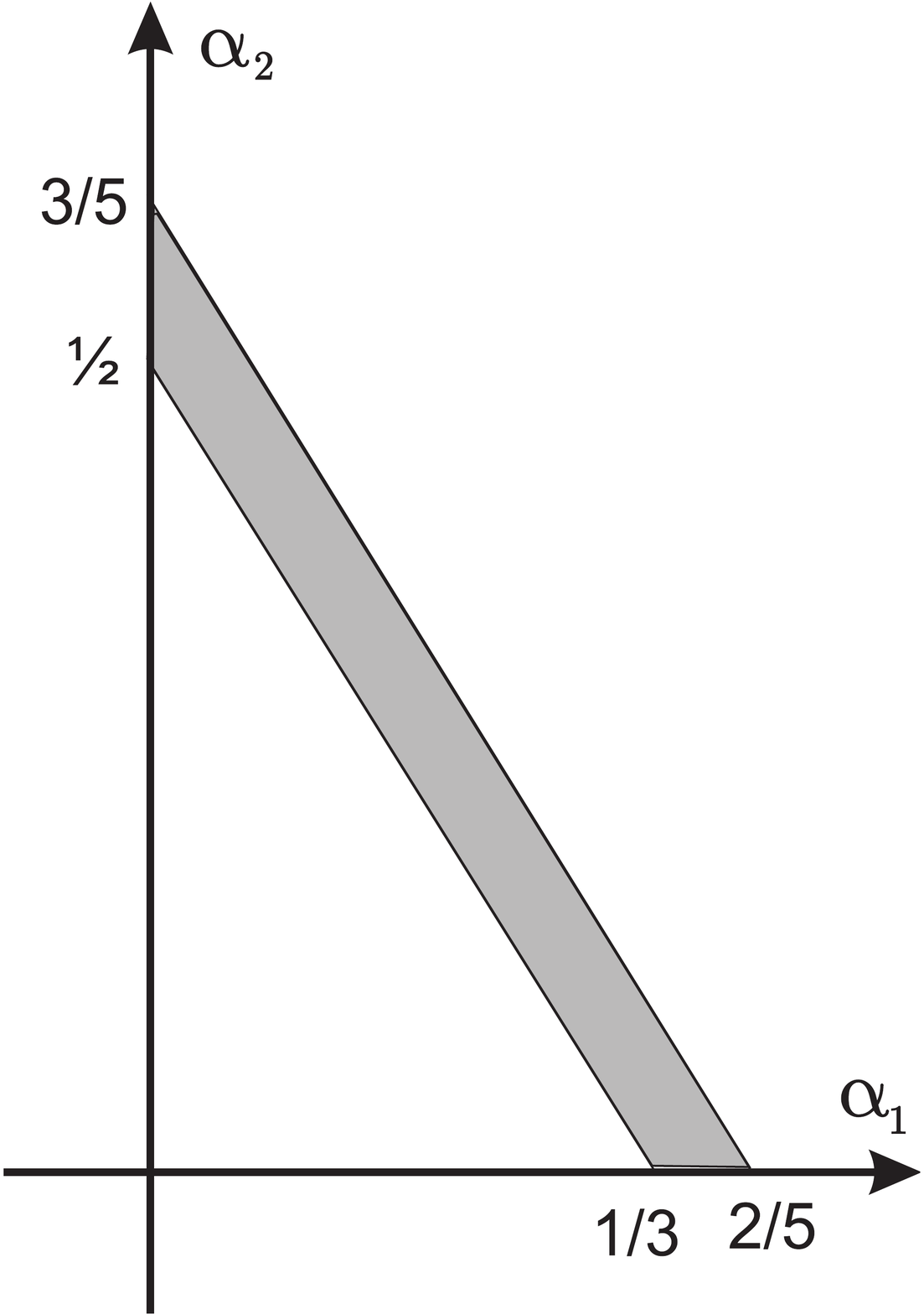} 
   \caption{The region on the plane of $\alpha_1, \alpha_2$ where the fundamental group $\pi_1(Y)$ has 2-tosion.}
   \label{fig:2torsion}
\end{figure}

\begin{proof}[Proof of Theorem \ref{2torsiona}]
Consider the triangulation $S_0$ of the real projective plane $P^2$ shown in Figure \ref{proj}; 
it has 6 vertices, 15 1-simplexes and 10 2-simplexes. It is the triangulation of $P^2$ having the smallest number of vertices. By Theorem 4.4 from \cite{CF14} $S_0$ is balanced, which means that for any non-empty subcomplex $T\subset S_0$ one has
$$\frac{f_i(T)}{f_0(T)} \le \frac{f_i(S_0)}{f_0(S_0)}.$$
Applying Theorem \ref{thm1}, {\bf B}, se see that for any $T\subset S_0$, $T\not=\emptyset$, 
\begin{eqnarray*}
\left[n^{f_0(T)}\prod_{i=0}^2 p_i^{f_i(T)} \right]^{\frac{1}{f_0(T)}} \ge 
n\cdot \prod_{i=0}^2 p_i^{\frac{f_i(T)}{f_0(T)}} 
\ge n\cdot \prod_{i=0}^2 p_i^{\frac{f_i(S_0)}{f_0(S_0)}} = np_0 p_1^{5/2}p_2^{5/3}\to \infty. 
\end{eqnarray*}
Thus we see that under the assumption (\ref{3}) the simplicial complex 
$S_0$ embeds into a random complex 
$Y\in \Omega_n^r$, a.a.s.

\begin{figure}[h] 
   \centering
   \includegraphics[width=3in]{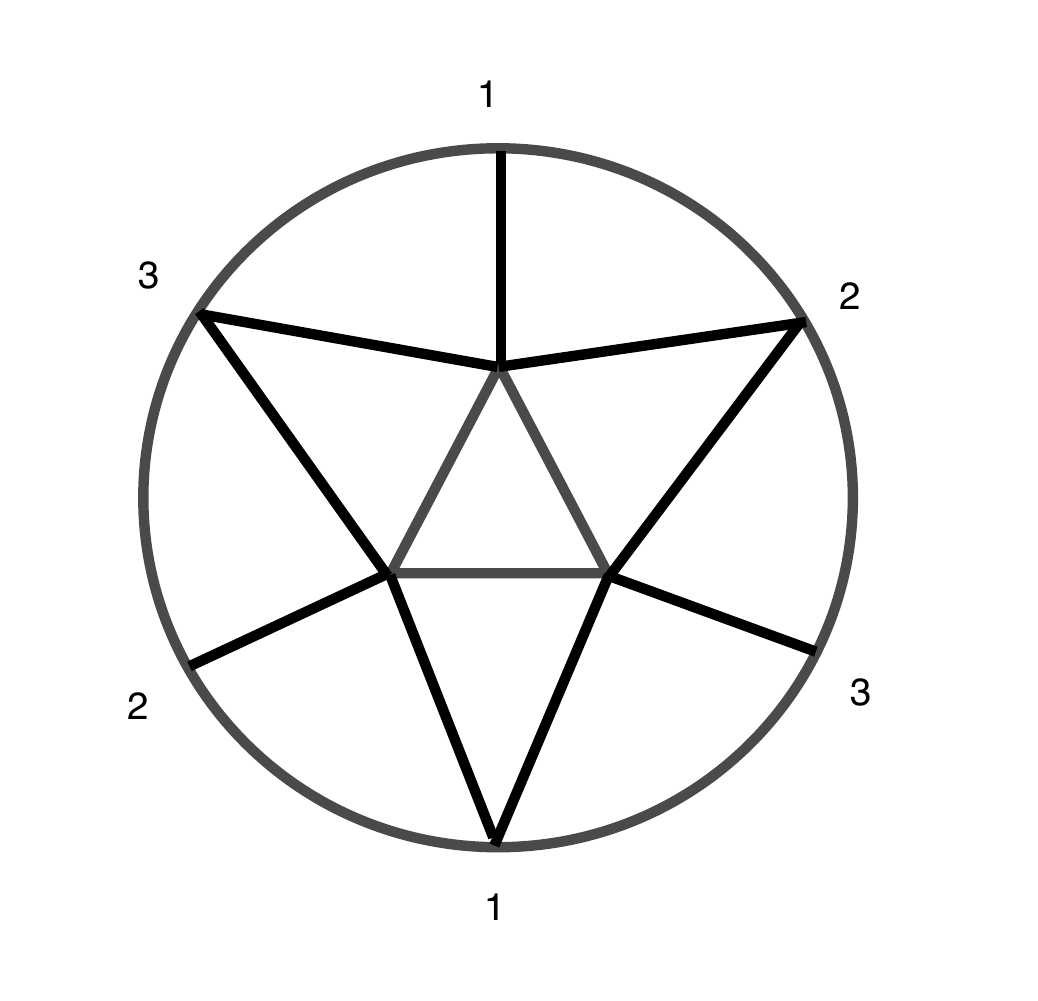} 
   \caption{The triangulation of the real projective plane having 6 vertices, 15 edges and 10 faces. The vertices and edges on the outer circle must be identified as indicated. }\label{proj} 
\end{figure}

We want to show that $Y$ contains $S_0$ 
as {\it \lq\lq an essential subcomplex\rq\rq},\,  i.e. such that the inclusion $S_0\subset Y$ induces a non-trivial homomorphism 
$\pi_1(S_0)=\Z_2\to \pi_1(Y)$; this would imply that $\pi_1(Y)$ has 2-torsion. 
We shall use Theorem \ref{thm3} to show the existence of an essential embedding $S_0\to Y$. Our strategy will be similar to those used in the proof of Theorem \ref{nontrivial}. 

We construct below a finite list $\cal L_\epsilon$ of 2-complexes $S$ containing $S_0$ such that every non-essential embedding of $S_0$ into $Y\in \Omega'_n$ extends to an embedding $S\to Y$, for some $S\in \cal L_\epsilon$. 

Consider the set $\cal L_\epsilon$ of isomorphism types of pairs
$(S,S_0)$ where $S$ is a finite 2-complex containing $S_0$ satisfying the following conditions:

(a) the inclusion $S_0\to S$ induces the trivial homomorphism of the fundamental groups;

(b) $S$ is minimal in the sense that for any proper subcomplex $S_0\subset S'\subsetneq S$ the inclusion $S_0\to S'$ induces an injective homomorphism 
$\pi_1(S_0)\to \pi_1(S')$; 
 
 (c) $f_2(S)\le 3\cdot c_\epsilon^{-1}+10$;
 
 (d) $S$ is $\epsilon$-admissible. 
 
 Note that any $S\in \mathcal L_\epsilon$ is pure. Indeed, if $S'\subset S$ denotes the pure part of $S$ then $\pi_1(S')\to \pi_1(S)$ is injective and hence 
 the inclusion $S_0\subset S'$ induces a trivial homomorphism $\pi_1(S_0)\to \pi_1(S')$; therefore the minimality property (b) implies $S=S'$.

 Consider the minimal cycles contained in a complex $S\in \cal L_\epsilon$. 
 If $Z\subset S$ is a minimal cycle of type A 
 then $Z$ cannot be contained in $S_0$ and by Lemma \ref{typeA} there is a 2-simplex $\sigma\subset Z- S_0$ with $\partial \sigma$ null-homotopic in $Z-\int(\sigma)$. Removing $\sigma$ from $S$ does not change the fundamental group and leads to a proper subcomplex $S_0\subset S'\subset S$ contradicting the minimality property (b). Hence $S$ does not contain minimal cycles of type A.
 
 Let $Z\subset S$ be a minimal cycle of type B where $S\in \cal L_\epsilon$. Let $Z_0\subset Z$ be the core of $Z$. Recall that the core $Z_0$ is either homeomorphic to $P^2$ or to the quotient $Q^2$ of $P^2$ where two adjacent edges are identified. If 
 $Z_0$ does not coincide with $S_0$ then by Lemma \ref{typeB} we may find a 2-simplex $\sigma\subset Z_0-S_0$ with $\partial \sigma$ null-homotopic in $Z-\int(\sigma)$. Removing $\sigma$ from $S$ does not change the fundamental group and leads to a subcomplex of $S$ contradicting the minimality property (b). Hence $Z_0=S_0$. 
 
 This shows that $S$ cannot have minimal cycles of type A and any minimal cycle of type B contained in $S$ must have $S_0$ as its core. 
 
 Let $Z\subset S$ be a minimal cycle of type B contained in $S$. Then $S_0\subset Z$and by Lemma \ref{typeB} $Z$ is simply connected and therefore the inclusion $S_0\subset Z$ induces the trivial homomorphism on the fundamental groups. Hence by minimality (b) we have $Z=S$. This shows that 
 $b_2(S)\le 1$. In particular, 
 \begin{eqnarray}\label{chirel}
 \chi(S, S_0) = \chi(S)-\chi(S_0) \le 1.
 \end{eqnarray}

 Next we show that $L(S,S_0)=L(S)-L(S_0)\leq -3$. 
 Let $S'=\overline{S-S_0}$ be the closure of the complement of $S_0$ in $S$. The intersection $S_0\cap S'$ is a graph $\Gamma$. 
 If $b_1(\Gamma)=0$ then the inclusion $S_0\to S$ is injective on the fundamental groups. Hence $b_1(\Gamma)\geq 1$ and thus $f_1(\Gamma)\geq 3$ implying that 
 \begin{eqnarray}
 L(S, S_0) = L(S)-L(S_0) \leq -3.
 \end{eqnarray}

Note that for any simplicial complex $X$ of dimension $\leq 2$ one has 
\begin{eqnarray}
3\chi(X)+L(X)=3f_0(X)-f_1(X), \nonumber \\ 
\label{qlrel}\\
2\chi(X)+L(X)=2f_0(X)-f_2(X).\nonumber
\end{eqnarray} 
Using these equalities and the above relations we obtain
\begin{align*}
3f_0(S,S_0)-f_1(S,S_0)&=3\chi(S,S_0)+L(S,S_0)\leq 0\\
2f_0(S,S_0)-f_2(S,S_0)&=2\chi(S,S_0)+L(S,S_0)\leq -1
\end{align*}
We conclude that for any $S\in \mathcal L_\epsilon$ one has
 \begin{eqnarray*}
 f_1(S,S_0) \, &\ge&\,  3f_0(S,S_0),\label{f12}\\
 f_2(S,S_0) \, &\ge &\, 2f_0(S,S_0)+1.\label{f22}
 \end{eqnarray*}

Now we see that
\begin{eqnarray}\label{prodrel} 
 n^{f_0(S,S_0)}\prod_{i=0}^2 p_i^{f_i(S, S_0)} \le \left[np_0p_1^3p_2^2\right]^{f_0(S,S_0)}\cdot p_2\to 0.
\end{eqnarray}
Applying Theorem \ref{thm3} we obtain that with probability tending to one, $Y\in \Omega_n^r$ admits an embedding of $S_0$ which does not extend to an embedding $S\to Y$ for every $S\in \mathcal L_\epsilon$. 

Let $c_\epsilon>0$ be the constant given by Theorem \ref{hyp}. Consider the set $\Omega'_n\subset \Omega_n^r$ consisting of complexes $Y\in \Omega_n^r$ satisfying the following three conditions:

(1)  $ Y$ satisfies $I(Y)\ge c_\epsilon$; 

(2) any 2-dimensional subcomplex $S\subset Y$ with $$f_2(S) \le 3\cdot c_\epsilon^{-1}+10$$ is $\epsilon$-admissible. 

(3) $Y$ contains a copy of $S_0$ as a subcomplex such that there exists no complex $S\in \mathcal L_\epsilon$ for which the embedding $S_0\to Y$ can be extended to an embedding $S\to Y$. 

By Theorem \ref{hyp}, Lemma \ref{small}, Theorem \ref{thm1} and Theorem 4.4 from \cite{CF14}, 
the probability that $Y$ belongs to $\Omega'_n$ tends to one as $n\to \infty$. 

 Let us explain that for $Y\in \Omega_n'$
 the embedding 
 $S_0\to Y$ is essential. 
For any embedding $S_0\subset Y$ which is not essential there exists a simplicial disc 
$D$ with at most $3\cdot c_\epsilon^{-1}$ 2-simplexes and $\partial D=C$ where $C\subset S_0$ is the non-null-homotpic loop on $S_0$ of length 3. 
The disc $D$
is simplicially mapped into $Y$ extending the embedding $C\to Y$ (this follows from the inequality $I(Y)\ge c_\epsilon^{-1}$ given by Theorem \ref{hyp}). 
The union of $S_0$ and the image of this disc is a subcomplex $S\subset Y$ containing $S_0$ which satisfies the properties (a), (c). 
Property (d) is satisfied because of Lemma 
\ref{small}. 
If the minimality property (b) is not satisfied, then instead of $S$ we can take an appropriate smaller subcomplex. 
Hence any non-essential embedding 
$C\to Y$ can be extended to an embedding $S\to Y$ for some $S\in \mathcal L_\epsilon$. 

This completes the proof.
\end{proof}

Note that the assumption $p_2\to 0$ was essentially used in the proof since in the product (\ref{prodrel}) the exponent $f_0(S, S_0)$ may happen to be $0$.
We believe that Theorem \ref{2torsiona} will remain true under a weaker assumption $p_2<c<1$ where $c$ is a constant.  

Next we state another result assuming that $p_2=1$, which is a generalisation of Theorem 7.2 from \cite{CF2}.

\begin{theorem} \label{2torsionb} 
Let $Y\in \Omega_n^r$ be a random simplicial complex with respect to the probability measure
$\PP_{r,\p}$ where $\p=(p_0, p_1, \dots, p_r)$. Assume that 
\begin{eqnarray}\label{2b}
p_2=1,
\end{eqnarray}
\begin{eqnarray}\label{3b}
np_0p_1^{30/11}\to \infty,
\end{eqnarray}
and for some $\epsilon>0$,
\begin{eqnarray}\label{4b}
(np_0)^{1+\epsilon}p_1^3\to 0.
\end{eqnarray}
Then  the fundamental group $\pi_1(Y)$ has nontrivial elements of order two, a.a.s.
\end{theorem}

\begin{remark}{\rm 
In the special case when $$p_i=n^{-\alpha_i}$$ with $\alpha_i\ge 0$ constant, where $i=0, 1, \dots$ Theorem \ref{2torsionb} states that the fundamental group of a random complex 
$Y$ has a nontrivial element of order 2 assuming that 
\begin{eqnarray*}
&&\alpha_0+\frac{30}{11} \alpha_1 \, <1, \\
&&\alpha_0 +3\alpha_1>1,\\
&&\alpha_2=0.
\end{eqnarray*}
This result was proven in \cite{CF2} in the special case when $\alpha_2=0$.

}
\end{remark}

\begin{remark} \rm 
The main distinctions between Theorems \ref{2torsiona} and \ref{2torsionb}
are the assumptions regarding the behaviour of $p_2$. 
The 2-torsion in the fundamental group is generated by essential embeddings of the real projective plane $P^2$. 
In the case of Theorem \ref{2torsiona} we are dealing with the embeddings of 
 the minimal triangulation $S_0$ of $P^2$. However, in the case when $p_2=1$ the triangular essential loop of $S_0$ bounds a triangle in $Y$. 
 This explains that in the case $p_2=1$ one has consider {\it clean} triangulations of 
$P^2$. Recall that a triangulation of a 2-complex is called clean if for any clique of 3 vertices $\{v_0, v_1, v_2 \}$
the complex contains also the 2-simplex $(v_0, v_1, v_2)$.

We shall use the following fact: any clean triangulation of the projective plane $P^2$ contains at least $11$ vertices and $30$ edges, see \cite{HR}. The minimal clean triangulation is shown in Figure \ref{cleanp2}; the antipodal points of the circle must be identified. 
\begin{figure}[h]
\centering
\includegraphics[width=0.4\textwidth]{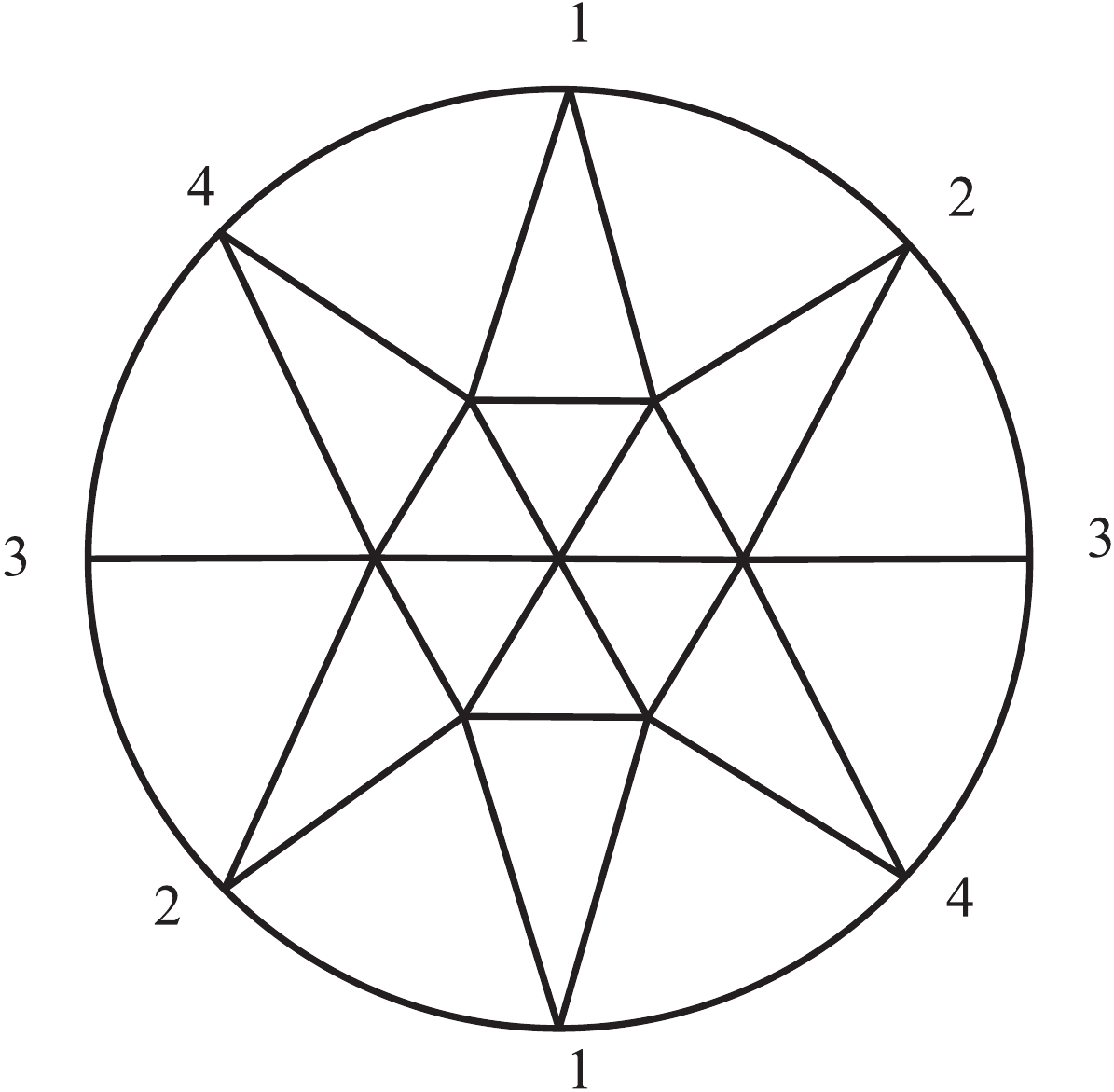}
\caption{The minimal clean triangulation of $P^2$, according to \cite{HR}.}\label{cleanp2}
\end{figure}

\end{remark}

\begin{proof}[Proof of Theorem \ref{2torsionb}.] The proof is very similar to the proof of Theorem \ref{2torsiona}; we shall indicate the main steps and emphasise the main distinctions. 

Let $S_0$ be the minimal clean triangulation of the real projective plane shown on Figure \ref{cleanp2}. It is balanced (by Theorem 4.4 from \cite{CF14}) and as in the proof of Theorem \ref{2torsiona} we find that for $np_0p_1^{30/11}p_2^{20/11} \to \infty$ the complex $S_0$ simplicially embeds into $Y$, a.a.s. 
Since we assume that $p_2=1$, this conditions coincides with (\ref{3b}). 

As above, we want to show that $Y$ contains $S_0$ 
as {\it \lq\lq an essential subcomplex\rq\rq}.

Consider the set $\cal L_\epsilon$ of isomorphism types of pairs
$(S,S_0)$ where $S$ is a finite 2-complex containing $S_0$ satisfying the following conditions:

(a) the inclusion $S_0\to S$ induces the trivial homomorphism of the fundamental groups;

(b) $S$ is minimal in the sense that for any proper subcomplex $S_0\subset S'\subsetneq S$ the inclusion $S_0\to S'$ induces an injective homomorphism 
$\pi_1(S_0)\to \pi_1(S')$; 
 
 (c) $f_2(S)\le 4\cdot c_\epsilon^{-1}+20$;
 
 (d) $S$ is $\epsilon$-admissible. 
 
 As in the proof of Theorem one obtains that each $S\in \mathcal L_\epsilon$ is pure and  
 \begin{eqnarray}\label{chirell}
 \chi(S, S_0) = \chi(S)-\chi(S_0) \le 1.
 \end{eqnarray}

Let us show that  
\begin{eqnarray}\label{llrel1}
 L(S, S_0) = L(S)-L(S_0) \leq -4.
 \end{eqnarray} 
 (unlike the case of Theorem \ref{2torsiona}). 
 Let $S'=\overline{S-S_0}$ be the closure of the complement of $S_0$ in $S$. The intersection $S_0\cap S'$ is a graph $\Gamma$. 
 If $b_1(\Gamma)=0$ then the inclusion $S_0\to S$ is injective on the fundamental groups. Hence $b_1(\Gamma)\geq 1$. Thus $\Gamma$ has a cycle and has at least 3 edges. However if $f_1(\Gamma)=3$ then the cycle of $\Gamma$ bounds a 2-simplex on $S_0$ and hence the inclusion $S_0\to S$ is injective on the fundamental group, in contradiction with our assumption (a). 
Thus, $f_1(\Gamma)\geq 4$ implying (\ref{llrel1}).

Combining (\ref{chirell}) and (\ref{llrel1}) with (\ref{qlrel}), 
we conclude that for any $S\in \mathcal L_\epsilon$ one has
 \begin{eqnarray*}
 f_1(S,S_0) \, &\ge&\,  3f_0(S,S_0)+1,\label{f122}\\
 f_2(S,S_0) \, &\ge &\, 2f_0(S,S_0)+2.\label{f222}
 \end{eqnarray*}
 and we see that
\begin{eqnarray}\label{prodrell} 
 n^{f_0(S,S_0)}\prod_{i=0}^2 p_i^{f_i(S, S_0)}& \le& \left[np_0p_1^3p_2^2\right]^{f_0(S,S_0)}\cdot p_1p_2^2 \\&=& \left[np_0p_1^3\right]^{f_0(S, S_0)}\cdot p_1\\ & \le& p_1\to 0.
\end{eqnarray}
Note that $p_1\to 0$ as follows from (\ref{4b}). 

The rest of the proof is identical to the proof of Theorem \ref{2torsiona}. 

This completes the proof.

\end{proof}

\section{Higher torsion}

In this section we show that random simplicial complexes have no odd torsion for a large range of probability parameters. 

\begin{theorem}\label{oddtorsion} Let $m\ge 3$ be a fixed odd prime. Consider a random simplicial complex $Y\in \Omega_n^r$, $r\ge 2$, with respect to the 
probability measure $\PP_{r, \p}$ where $\p=(p_0, p_1, \dots, p_r)$. 
Assume that for some $\epsilon>0$ one has
\begin{eqnarray*}
(np_0)^{1+\epsilon} p_1^3p_2^2\to 0.
\end{eqnarray*}
Then a random complex $Y\in\Omega_n^r$ with probability tending to 1 has the following property: the fundamental group of any connected subcomplex $Y'\subset Y$ has no $m$-torsion.
\end{theorem}

\begin{remark}{\rm 
In the special case when $$p_i=n^{-\alpha_i}$$ with $\alpha_i\ge 0$ being constant, where $i=0, 1, \dots$ Theorem \ref{oddtorsion} states that 
for any odd prime $m\ge 3$ 
the fundamental group of a random complex 
$Y$ has no nontrivial elements of order $m$ assuming that 
\begin{eqnarray*}
&&\alpha_0 +3\alpha_1+2\alpha_2>1.
\end{eqnarray*}
This result was proven in \cite{CF1} in the special case when $\alpha_1=0$ and in \cite{CF2} in the special case $\alpha_2=0$. 

}
\end{remark}

The proof of Theorem \ref{oddtorsion} given below uses an auxiliary material which we now describe.

Let $f_m: S^1\to S^1$ denote the canonical degree $m$ map, $f_m(z)=z^m$ where $z\in S^1$;
we think of $S^1$ as being the unit circle on the complex plane. 
Any simplicial complex $\Sigma$ homeomorphic 
$$M(\Z_m, 1)=S^1\cup_{f _m}e^2$$
is called {\it a Moore surface}. 

Everywhere in this section we shall assume that $m\ge 3$ is a fixed odd prime. 

Then any Moore surface $\Sigma$ 
 has a well defined circle $C\subset \Sigma$ (called {\it the singular circle}) which is the union of all edges of degree $m$; 
 all other edges of
$\Sigma$ have degree $2$. Clearly, the homotopy class of the singular circle generates the fundamental group $\pi_1(\Sigma)\simeq \Z_m$.

Define an integer $N_m(Y)\ge 0$ associated to any connected simplicial complex $Y\in \Omega_n^r$. If $\pi_1(Y)$ has no $m$-torsion we set $N_m(Y)=0.$
If $\pi_1(Y)$ has elements of order $m$ we consider homotopically nontrivial simplicial maps
$\gamma: C_r \to Y$,
where $C_r$ is the simplicial circle with $r$ edges, such that
\begin{enumerate}
  \item[(a)] $\gamma^m$ is null-homotopic (as a free loop in $Y$);
  \item[(b)] $r$ is minimal: for $r'<r$ any simplicial loop $\gamma:C_{r'} \to Y$ satisfying (a) is homotopically trivial.
  \end{enumerate}
Any such simplicial map $\gamma:C_r \to Y$ can be extended to a simplicial map $f: \Sigma \to Y$ of a Moore surface $\Sigma$ such that the singular circle $C$ of $\Sigma$ is isomorphic to $C_r$ and $f|C=\gamma$. We shall say that a simplicial map $f:\Sigma \to Y$ is {\it $m$-minimal} if
it satisfies (a), (b) and the number of 2-simplexes in $\Sigma$ is the smallest possible. Clearly, any m-minimal map $\Sigma\to Y$ induces an injective homomorphism 
$\pi_1(\Sigma)\simeq \Z_m\to \pi_1(Y)$. 
We denote by
$$N_m(Y)\in \Z$$
the number of 2-simplexes in a triangulation of a Moore surface $\Sigma$ admitting an $m$-minimal map $f: \Sigma \to Y$.


We recall Lemma 4.7 from \cite{CF1}: 
\begin{lemma}\label{lmc}
Let $Y$ be a simplicial complex satisfying $I(Y)\ge c>0$ and let 
$m\ge 3$ be an odd prime.
Then
one has
$$N_m(Y) \le \left(\frac{6m}{c}\right)^2. $$
\end{lemma}

Using this lemma we may obtain a global upper bound on the numbers $N_m(Y)$ for random complexes:

\begin{theorem}\label{thmm}
Assume that the probability multi-parameter $\p=(p_0, p_1, \dots, p_r)$, where $r\ge 2$, satisfies 
$$ (np_0)^{1+\epsilon}p_1^3p_2^2\to 0$$ with $\epsilon >0$ is fixed.
Let $m\ge 3$ be an odd prime.
Then there exists a constant $C_\epsilon>0$ such that a random complex $Y\in \Omega_n^r$ with probability tending to 1 has the following property: for any subcomplex $Y'\subset Y$ one has
\begin{eqnarray}\label{ineqq}
N_m(Y') \le C_\epsilon.
\end{eqnarray}
\end{theorem}
\begin{proof}  We know from Theorem \ref{hyp} that, with probability tending to 1, a random complex $Y$ has the following property:
for any subcomplex $Y'\subset Y$ one has
$I(Y')\ge c_\epsilon>0$ where $c_\epsilon>0$ is the constant given by Theorem \ref{hyp}.
Then, setting $C= \left(\frac{6m}{c_\epsilon}\right)^2$, the inequality (\ref{ineqq}) follows from Lemma \ref{lmc}.
\end{proof}

Now we are ready to present the proof of Theorem \ref{oddtorsion}:

\begin{proof}[Proof of Theorem \ref{oddtorsion}] Let $c_\epsilon>0$ be the number given by Theorem \ref{hyp}.
Consider the finite set of all isomorphism types of Moore surfaces $\mathcal S_m =\{\Sigma\}$ having at most
$\left(\frac{6m}{c_\epsilon}\right)^2$ two-dimensional simplexes. Let $\mathcal X_m$ denote
the set of isomorphism types of images of all surjective simplicial maps $\Sigma \to X$ inducing injective homomorphisms $\pi_1(\Sigma)=\Z_m \to \pi_1(X)$,
where $\Sigma \in \mathcal S_m$.
The set $\mathcal X_m$ is also finite.

From Theorem \ref{thmm} we obtain that, with probability tending to one, for any subcomplex $Y'\subset Y$, either $\pi_1(Y')$ has no $m$-torsion, or
there exists an $m$-minimal map $f: \Sigma \to Y'$ where $\Sigma$ is a Moore surface having at most $\left(\frac{6m}{c_\epsilon}\right)$ 2-simplexes of dimension 2; in the second case the image
$X=f(\Sigma)$ is a subcomplex of $Y'$ and $f:\Sigma \to X$ induces a monomorphism $\pi_1(\Sigma)\to \pi_1(X)$, i.e.
$X\in {\mathcal X}_m$. Denote by $\mathcal X'_m\subset \mathcal X_m$ the set of complexes $X\in\mathcal X_m$ such that their 2-dimensional pure 
parts are 
$\epsilon$-admissible. 
We may apply Lemma \ref{small}  (using the upper bound $f_2(X)\leq C=\left(\frac{6m}{c_\epsilon}\right)$)
to conclude that the pure part of the image $X=f(\Sigma)$ of any $m$-minimal map $f:\Sigma \to Y$ belongs to $\mathcal X'_m$. 
However, by Theorem \ref{wedge} the fundamental group of any $X\in \mathcal X'_m$ is a free product of several copies of $\Z$ and $\Z_2$ and hence it has no $m$-torsion. This leads to a contradiction which shows that the fundamental group of any subcomplex $Y'\subset Y$ does not have $m$-torsion, a.a.s.

\end{proof}

\section{Asphericity and the Whitehead Conjecture}

 
 Recall that a connected simplicial complex $Y$ is said to be {\it aspherical} if $\pi_k(Y)=0$ for all $k\ge 2$. 
 A 2-dimensional connected simplicial complex $Y$ is aspherical if and only if $\pi_2(Y)=0$. 
 The well-known Whitehead Conjecture states that any subcomplex of an aspherical 2-complex is also aspherical, see \cite{Adams, B, CC, R}. 
 At the time of writing the Whitehead Conjecture is still open. 
 
 In this section we show that for random simplicial 2-complexes, with probability tending to one, 
 the asphericity of a subcomplex of a random complex is equivalent to the absence of {\it \lq\lq small bubbles\rq\rq}. It follows that in the random setting any subcomplex of an aspherical 2-complex is also aspherical, supporting a probabilistic analogue of the Whitehead Conjecture.

\begin{theorem}\label{aspherical} Consider a random simplicial complex $Y\in \Omega_n^r$, $r\ge 2$, with respect to the probability measure 
$\PP_{r, \p}$ where $\p=(p_0, p_1, \dots, p_r)$. 
Assume that for some $\epsilon>0$ one has
$$(np_0)^{1+\epsilon}p_1^3p_2^2\, \to \, 0.$$
Let $c_\epsilon$ be the constant given by Theorem \ref{hyp}.
Then, a random complex $Y\in\Omega_n^r$ has the following property with probability tending to 1 as $n\to \infty$: 
any subcomplex $Y'\subset Y^{(2)}$ is aspherical if and only if every pure subcomplex $S\subset Y'$ satisfying $$f_2(S)\leq C= 16^2c_\epsilon^{-2}$$ 
is collapsible to a graph. In particular, under the above assumptions, with probability tending to one, any aspherical subcomplex of $Y^{(2)}$ satisfies the Whitehead Conjecture.
\end{theorem}
%

%
%

The proof given below in this section will use the following auxiliary material. 

Let $Y$ be a simplicial complex with $\pi_2(Y)\not=0$. As in \cite{CF1} and \cite{CF2}, we define a numerical invariant $M(Y)\in \mathbb{Z}$, $M(Y) \ge 4$, as the minimal number of faces in a
2-complex $\Sigma$
homeomorphic to the sphere
$S^2$ such that there exists a homotopically nontrivial simplicial map $\Sigma\to Y$.

We define $M(Y)=0$, if $\pi_2(Y)=0$.

\begin{lemma}[See Corollary 5.4 in \cite{CF1}]\label{lmm} Let $Y$ be a 2-complex with $I(Y)\ge c>0$. Then
$$M(Y)\le \left(\frac{16}{c}\right)^2.$$
\end{lemma}
Combining this lemma with Theorem \ref{hyp} we obtain:

\begin{lemma}\label{M}
Assume that $$(np_0)^{1+\epsilon}p_1^3p_2^2\to 0$$ for some $\epsilon>0$. 
 Then there exists a constant $C_\epsilon>0$ such that a random complex $Y\in\Omega_n^r$ has the following property with probability tending to one: for any subcomplex
$Y'\subset Y^{(2)}$ one has $$M(Y') \le C_\epsilon.$$
\end{lemma}

Hence, if we want to find homotopically nontrivial simplicial maps from $S^2$ to a random complex it is sufficient to consider triangulations of $S^2$ 
having at most $C_\epsilon$ 2-simplexes. 

The following lemma will be used in the proof of Theorem \ref{aspherical}. 

\begin{lemma}\label{lmccc}
Let $S$ be a connected admissible $2$-complex (see Definition \ref{eps}). Then $S$ is aspherical if and only if $S$ is simplicially collapsible to a graph.
\end{lemma}

\begin{proof} Let $S$ be a connected admissible $2$-complex satisfying $\pi_2(S)=0$.
Performing all possible simplicial collapses of 2-simplexes we may find a closed pure subcomplex $S'\subset S$ 
without free edges. We need to show that
$S'$ is 1-dimensional. We shall assume below that $\dim S'=2$ and arrive to a contradiction. 

If $b_2(S')\ge 1$ then $S'$ contains an admissible minimal cycle $Z\subset S'$ as a subcomplex. Using Lemmas \ref{typeA} and \ref{typeB} (depending on whether $Z$ is of type A or B) we find a 2-simplex $\sigma\subset Z$ such that $\partial \sigma$ is null-homotopic in $Z-\int(\sigma)$. Then $S'$ is homotopy equivalent to the wedge $S'-\int(\sigma)\vee S^2$ contradicting the assumption $\pi_2(S')=0$. 

Hence we must assume that $b_2(S')=0$. Using Corollary \ref{b2=0} we see that every strongly connected component of $S'$ is homeomorphic to either $P^2$ or to the quotient $Q^2$ of $P^2$ with two adjacent edges identified. In both cases we may apply a theorem of Cockcroft \cite{CC} (see also Adams \cite{Adams}) which 
claims that $\pi_2(S')\neq 0$, in contradiction with our hypothesis. 
\end{proof}

\begin{lemma}\label{lmccc1}
Let $S$ be a connected admissible $2$-complex. If $S$ is not aspherical then $S$ contains a subcomplex which is homotopy equivalent to one of 
$S^2$, $S^2\vee S^1$ or $P^2$. 
\end{lemma}

\begin{proof} Let $S$ be a connected admissible $2$-complex; without loss of generality we may assume that $S$ is closed and pure. If $b_2(S)\not=0$ then 
$S$ contains an admissible minimal cycle $Z\subset S$ and $Z$ is homotopy equivalent to either $S^2$, $S^2\vee S^1$ (in the case of type A, see Lemma \ref{typeA}), or to $P^2$ (in the case of type B, see Lemma \ref{typeB}). 
In the case $b_2(S)=0$ we invoke Corollary \ref{b2=0}. 
\end{proof}

\begin{proof}[Proof of Theorem \ref{aspherical}] Consider the set $\mathcal S$ of isomorphism types of all pure 2-comp\-lexes $S$ satisfying $f_2(S)\le C$. 
We may represent $\mathcal S$ as the disjoint union $\mathcal S = \mathcal S'\cup \mathcal S''$ where the complexes $S\in \mathcal S'$ are $\epsilon$-admissible 
and the complexes $S\in \mathcal S''$ are not. 

Let $\Omega_n'\subset \Omega_n^r$ be the set of complexes $Y\in \Omega_n^r$ such that (a) no $S\in \mathcal S''$ can be embedded into $Y$ and (b)
each $Y\in \Omega'_n$ satisfies the conclusion of Lemma \ref{M}. 
By Lemmas \ref{small} and \ref{M} we know that $\PP_{r, \p}(\Omega_n')\to 1$ as $n\to \infty$.

Suppose that $Y\in \Omega_n'$ and 
let $Y'\subset Y^{(2)}$ be an aspherical subcomplex. 
For any pure subcomplex $S\subset Y'$  with $f_2(S)\leq C$ we know that $S\in \mathcal S'$ and by Lemma \ref{lmccc1} either $S$ is aspherical or it contains a subcomplex homotopy equivalent to either $S^2$ or $P^2$. Both these possibilities would imply $\pi_2(Y')\not=0$ (the case of $S^2$ is obvious and the case of $P^2$ follows from the work of Cockcroft \cite{CC}, see also \cite{Adams}. Thus we see that in $Y'\subset Y$ is aspherical then any 
subcomplex $S\subset Y'$ with $f_2(S)<C$ is also aspherical; the latter due to Lemma \ref{lmccc} is equivalent for $S$ to be collapsible to a graph.

We now prove the inverse implication by assuming that $Y'\subset Y$ is non-aspherical and $Y\in \Omega_n'$. 
By Lemma \ref{M} there exists a 2-complex 
$\Sigma$ homeomorphic to a 2-sphere with $f_2(\Sigma)\leq C$ and a homotopically nontrivial simplicial map $\phi:\Sigma\to Y'$. 
We denote $S=\phi(\Sigma)$ and thus we have $f_2(S)\leq f_2(\Sigma)\leq C$ and $\pi_2(S)\neq 0$. 
Hence we conclude that if $Y'\subset Y$ is not aspherical then there exists a subcomplex $S\subset Y'$, $S\in \mathcal S'$, and $f_2(S)\leq C$ and $\pi_2(S)\neq 0$. 

This completes the proof.
\end{proof}

\section{Geometric and cohomological dimension of the fundamental group of a random simplicial complex}

\begin{theorem}\label{cd2} Consider a random simplicial complex $Y\in \Omega_n^r$ with respect to the multi-parameter probability measure $\PP_{r, \p}$, where $\p=(p_0, p_1, \dots, p_r)$, $r\ge 2$. Assume that
\begin{eqnarray}\label{projembeds}
np_0p_1^{5/2}p_2^{5/3}\to 0
\end{eqnarray}
Then for any choice of the base point $y_0\in Y$ the fundamental group $\pi_1(Y, y_0)$ has
geometric dimension at most $2$, a.a.s. In particular, the group
$\pi_1(Y,y_0)$ has cohomological dimension at most 2 and is torsion free, a.a.s.
\end{theorem}

\begin{proof} We will show that with probability tending to one a random complex $Y\in \Omega_n^r$ contains a 2-dimension subcomplex 
$Y'\subset Y$ such that 
\begin{itemize}
\item[(1)] $Y^{(1)}=Y'^{(1)}$; 

\item[(2)]  $\pi_1(Y',y_0)\simeq \pi_1(Y, y_0)$ for any vertex $y_0\in Y$;

\item[(3)] any connected component of $Y'$ is aspherical. 
\end{itemize}
This would clearly imply the statement of Theorem \ref{cd2}. 

Note that our assumption (\ref{projembeds}) implies the condition (\ref{epsilon}) for any $\epsilon < 1/5$. 
Hence we may apply Theorem \ref{hyp} and Theorem 
\ref{aspherical}. To be specific we may set $\epsilon =1/10$ and denote by $c_\epsilon>0$ the constant given by Theorem \ref{hyp}. 

Denote by $\Omega'_n\subset \Omega_n^r$ the set of complexes $Y\in \Omega_n^r$ satisfying the following conditions: 
\begin{itemize}

\item[(a)]  For $Y\in \Omega'_n$ any 2-dimensional pure subcomplex $S\subset Y$ with $f_2(S) \le C$ is admissible. Here $C$ denotes $16^2 c_\epsilon^2$.

\item[(b)] For $Y\in \Omega'_n$ a subcomplex $Y'\subset Y^{(2)}$ is aspherical if and only if every subcomplex $S\subset Y'$ satisfying $f_2(S) <C$ is collapsible to a graph. 

\item[(c)] Any complex $Y\in \Omega'_n$ has no closed admissible 2-dimensional pure subcomplexes $S\subset Y$ with $f_2(S)\le C$ satisfying $b_2(S)=0$. 
\end{itemize}
We know that $\PP_{r, \p}(\Omega'_n) \to 1$ due to 
Lemma \ref{small} and Theorem \ref{aspherical}; to explain that the property (c) can be achieved we observe that by Corollary \ref{b2=0} any such $S$ is either a triangulation of $P^2$ or the quotient $Q^2$ of $P^2$ obtained by identifying two adjacent edges in certain triangulation. Assuming that $S$ is a triangulation of $P^2$ we have 
$f_0(S)-f_1(S)+f_2(S)=1$, $3f_2(S)=2f_1(S)$ and $f_0(S)\ge 6$ implying that $f_1(S)\ge 15$ and $f_2(S)\ge 10$ and also
$$\frac{f_0(S)}{f_1(S)} = \frac{1}{3} + \frac{1}{f_1(S)} \, \le\,  2/5,$$
$$\frac{f_0(S)}{f_2(S)} = \frac{1}{2} + \frac{1}{f_2(S)} \, \le\,  3/5.$$
Therefore, 
$$np_0p_1^{\frac{f_1(S)}{f_0(S)}}p_2^{\frac{f_2(S)}{f_0(S)}}\, \le\,  np_0p_1^{5/2}p_2^{5/3} \to 0$$
because of our assumption (\ref{projembeds}). In the case when $S$ is obtained from a triangulation $T$ of $P^2$ by identifying two adjacent edges we have
$f_0(S)=f_0(T)-1$, $f_1(S)=f_1(T)-1$ and $f_2(S)=f_2(T)$. Then 
$$\frac{f_i(T)}{f_0(T)} \, \ge\,  \frac{f_i(S)}{f_0(S)}, \quad\quad  i=1, 2,$$
and we obtain
$$np_0p_1^{\frac{f_1(S)}{f_0(S)}}p_2^{\frac{f_2(S)}{f_0(S)}}\, \le\,  np_0p_1^{\frac{f_1(T)}{f_0(T)}}p_2^{\frac{f_2(T)}{f_0(S)}} \to 0$$
as shown above. Hence our statement follows by invoking Theorem \ref{thm1}. 

Given a complex $Y\in \Omega_n'$ consider an admissible minimal cycle $Z\subset Y$ with $f_2(Z)\le C$. 
Using Lemma \ref{typeA} and Lemma \ref{typeB} we may find a 2-simplex $\sigma \subset Z\subset Y$ such that removing it we do not change the fundamental group. Therefore we may inductively obtain a sequence of subcomplexes $$Y_0=Y^{(2)} \supset Y_1\supset Y_2 \dots\subset Y_N$$ with each complex $Y_{i+1}$ obtained from the previous $Y_i$ by removing the interior 
of a 2-simplex $\sigma\subset Y_i$ such that $\partial \sigma$ is null-homotopic in $Y_i-\int(\sigma)$. Let $Y'=Y_N\subset Y$ be the final complex in this sequence. 

We claim that the obtained complex $Y'$ is aspherical. Let $S\subset Y'$ be a pure subcomplex with $f_2(S)\le C$. 
Then $b_2(S)=0$ since otherwise $Y'$ would contain an admissible minimal cycle $Z$ with $f_2(Z)\le C$ contradicting our construction (here we use (a)). 
By (c) the complex $Y'$ contains no closed admissible 2-dimensional pure subcomplexes $S\subset Y$ with $f_2(S)\le C$ satisfying $b_2(S)=0$. Hence any 
subcomplex $S\subset Y'$ with $f_2(S)\le C$ is collapsible to a graph. By property (b) the complex $Y'$ is aspherical. 

This completes the proof. 

\end{proof}

Next we state an analogue of Theorem \ref{cd2} in the special case when $p_2=1$. 
\begin{theorem}\label{cd22}
Consider a random simplicial complex $Y\in \Omega_n^r$ with respect to the multi-parameter probability measure $\PP_{r, \p}$, where $\p=(p_0, p_1, \dots, p_r)$. Assume that $r\ge 2$ and $p_2=1$ and besides,
\begin{align}\label{3011}
np_0p_1^{30/11}\to 0. 
\end{align}
Then for any choice of the base point $y_0\in Y$ the fundamental group $\pi_1(Y, y_0)$ has
geometric dimension at most $2$, a.a.s. In particular, the group
$\pi_1(Y,y_0)$ has cohomological dimension at most 2 and is torsion free, a.a.s.
\end{theorem}

\begin{proof} The proof is similar to the proof of Theorem \ref{cd2}. The only difference is that in the case when $p_2=1$ we have a different requirement 
(\ref{3011}) on the absence of  embeddings of {\it clean} triangulations of the real projective plane $P^2$ into a random 2-complex. 
\end{proof}
\begin{figure}[h] 
   \centering
   \includegraphics[width=3in]{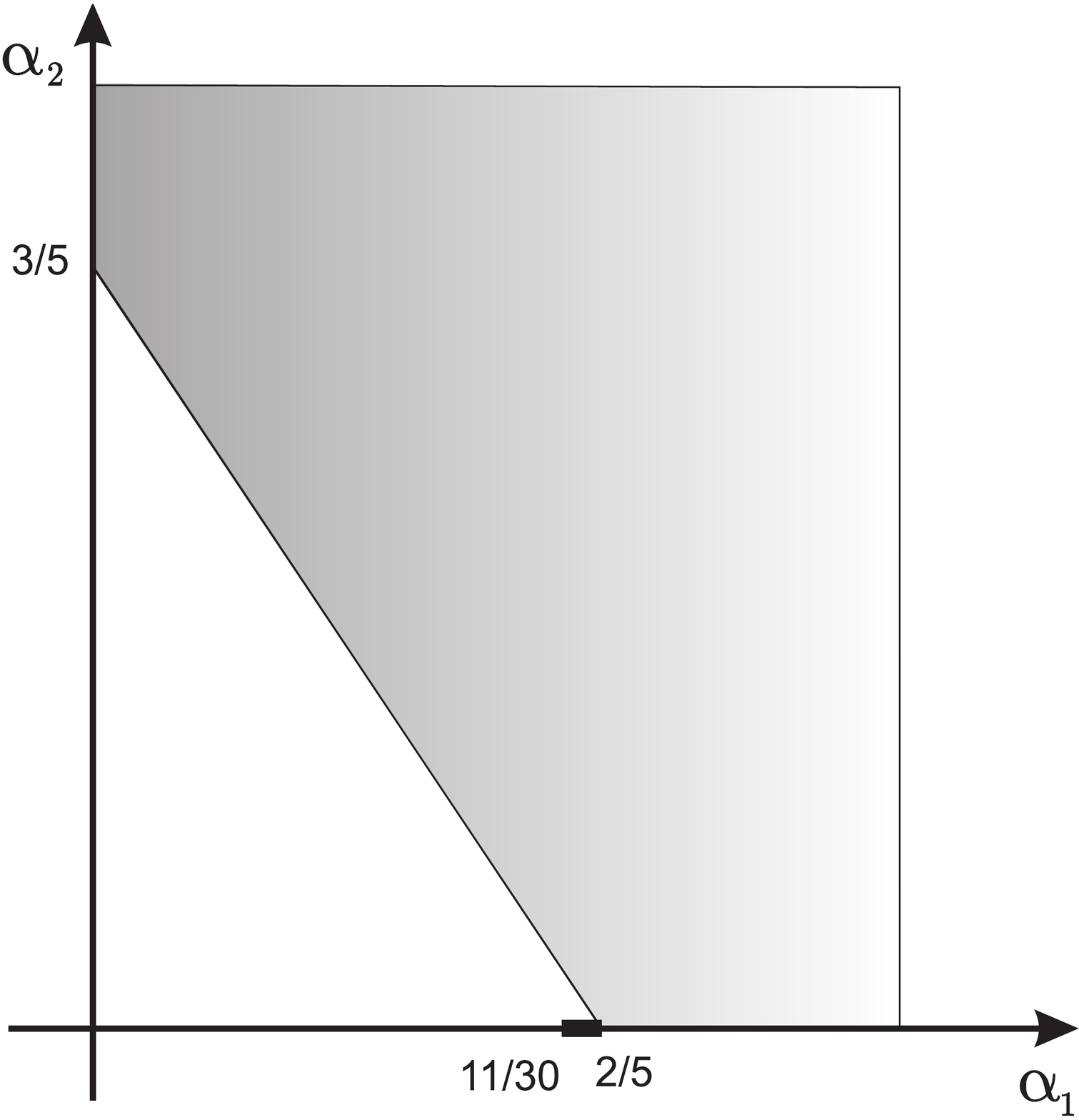} 
   \caption{The region on the plane of $\alpha_1, \alpha_2$ where the geometric dimension of the fundamental group $\pi_1(Y)$ equals 2.}
   \label{fig:cd2}
\end{figure}

Theorems \ref{cd2} and \ref{cd22} complement Theorems \ref{2torsiona} and \ref{2torsionb} about the existence of 2-torsion. 

\begin{remark}{\rm 
In the special case when $p_i=n^{-\alpha_i}$ with $\alpha_i\ge 0$ constants, where $i=0, 1, \dots$ Theorem \ref{cd2} states that the fundamental group of a random complex 
$Y$ has geometric dimension at most 2 assuming that either
\begin{eqnarray*}
&&\alpha_0+\frac{5}{2} \alpha_1 + \frac{5}{3} \alpha_2>1, 
\end{eqnarray*}
or 
\begin{eqnarray*} \alpha_0+\frac{30}{11} \alpha_1 >1, \quad \mbox{and}\quad \alpha_2=0.
\end{eqnarray*}
See Figure \ref{fig:cd2}; the corresponding set is represented by the shaded area together with an open interval on the $\alpha_1$ axis. 

These results were proven in \cite{CF1} in the special case when $\alpha_1=0$ and in \cite{CF2} in the case when $\alpha_2=0$. }
\end{remark}

\section{Appendix: Proof of Theorem \ref{uniform}}\label{app}

\subsection{Proof of Theorem \ref{uniform}}\label{secthm10}

\begin{definition} (See {\rm \cite{CF1}})
We say that a finite 2-complex $X$ is tight if for any proper subcomplex $X'\subset X$, $X'\not= X$, one has
$I(X') > I(X).$
\end{definition}
Clearly, one has
\begin{eqnarray}\label{ineq}
I(X) \ge \min \{I(Y) ; \mbox{\ $Y\subset X$ is a tight subcomplex}\}.
\end{eqnarray}
 By (\ref{ineq}) it is obvious that it is enough to prove Theorem \ref{uniform} under the additional assumption that $X$ is tight.

\begin{remark}\label{rmrk} {\rm Suppose that $X$ is pure and tight and suppose that $\gamma: S^1\to X$ is a simplicial loop with the ratio
$|\gamma|\cdot A_X(\gamma)^{-1}$ less than the minimum of the numbers $I(X')$ where $X'\subset X$ is a proper subcomplex. Let $b: D^2\to X$ be a minimal spanning disc for $\gamma$; then $b(D^2)=X,$ i.e. $b$ is surjective. Indeed, if the image of $b$ does not contain a 2-simplex $\sigma$ then removing
it we obtain a subcomplex $X'\subset X$ with $A_{X'}(\gamma)=A_X(\gamma)$ and hence $I(X') \le I(X)\le |\gamma|\cdot A_X(\gamma)^{-1}$ contradicting the assumption on $\gamma$. }
\end{remark}

\begin{lemma}\label{lm13} If $X$ is a admissible tight complex then $b_2(X)=0$.
\end{lemma}
\begin{proof} Assume that $b_2(X)\not=0$. Then there exists a admissible minimal cycle $Z\subset X$. Hence, by Lemmas \ref{typeA} and \ref{typeB} we may find a 2-simplex $\sigma\subset Z\subset X$
such that $\partial \sigma$ is null-homotopic in $Z-\sigma\subset X-\sigma=X'$. Note that $X'^{(1)}=X^{(1)}$ and a simplicial curve $\gamma: S^1\to X'$ is null-homotopic in $X'$ if and only if it is null-homotopic in $X$. Besides, $A_X(\gamma) \le A_{X'}(\gamma)$ and hence
$$\frac{|\gamma|}{A_X(\gamma)}\ge \frac{|\gamma|}{A_{X'}(\gamma)},$$
which implies that $I(X) \ge I(X')$. We obtained a contradiction since $X$ is tight.
\end{proof}

\begin{lemma}\label{lneg} Given $\epsilon >0$ there exists a constant $C'_\epsilon>0$ such that for every finite pure tight connected $\epsilon$-admissible complex $X$ satisfying $L(X) \le 0$ one has $I(X) \ge C'_\epsilon$.
\end{lemma}

This lemma is similar to Theorem \ref{uniform} but it has an additional assumption that $L(X) \le 0$. The assumption $L(X)\le 0$ can be replaced, without altering the proof, by any assumption of the type $L(X) \le 1000$, i.e. by any specific upper bound.

%

\begin{proof}[Proof of Lemma \ref{lneg}] We show that the number of isomorphism types of complexes $X$ satisfying the conditions of the lemma is finite; hence the proof follows  by setting $C'_\epsilon=\min I(X)$ and using Theorem \ref{wedge} which gives $I(X)>0$ (since $\pi_1(X)$ is hyperbolic) and hence $C'_\epsilon>0$.

By Lemma \ref{mu1 or mu2} we obtain
$$\mu_1(X')\geq \frac{1+\epsilon}{3}\quad \rm{or}\quad \mu_2(X')\geq \frac{1+\epsilon}{2}.$$
The inequality
$$\mu_1(X) = \frac{1}{3} +\frac{3\chi(X)+L(X)}{3f_1(X)} \ge \frac{1 + \epsilon}{3}$$
is equivalent to
$$f_1(X) \le \epsilon^{-1}\cdot (3\chi(X) +L(X)),$$
where $f_1(X)$ denotes the number of 1-simplexes in $X$.
By Lemma \ref{lm13} we have $\chi(X) =1-b_1(X) \le 1$ and using the assumption $L(X) \le 0$ we obtain
$f_1(X) \le \epsilon^{-1}.$
This implies the finiteness of the set of possible isomorphism types of $X$ and the result follows.

The case $\mu_2(X)\geq (1+\epsilon)/2$ is analogous.
\end{proof}

We will use a relative isoperimetric constant $I(X, X')\in \R$ for a pair consisting of a finite 2-complex
$X$ and its subcomplex $X'\subset X$; it is defined as the infimum of all ratios
$
{|\gamma|}\cdot{A_X(\gamma)}^{-1}$
where $\gamma: S^1\to X'$ runs over simplicial loops in $X'$ which are null-homotopic in $X$.
Clearly, $I(X,X')\ge I(X)$ and $I(X, X')=I(X)$ if $X'=X$.
Below is a useful strengthening of Lemma \ref{lneg}.

\begin{lemma}\label{lnegs} Given $\epsilon >0$, let $C'_\epsilon>0$ be the constant given by Lemma \ref{lneg}. Then for any finite pure, tight, $\epsilon$-admissible and connected 2-complex and for a connected subcomplex $X'\subset X$ satisfying $L(X') \le 0$ one has $I(X,X') \ge C'_\epsilon$.
\end{lemma}

\begin{proof}
We show below that under the assumptions on $X$, $X'$ one has
\begin{eqnarray} \label{y}
I(X,X') \ge \min_Y I(Y)
\end{eqnarray}
where $Y$ runs over all subcomplexes $X'\subset Y\subset X$ satisfying $L(Y)\le 0$. Clearly, any such $Y$ is $\epsilon$-admissible.
 By Lemma \ref{lm13} we have that $b_2(X)=0$ which implies that $b_2(Y)=0$. Besides, without loss of generality we may assume that
$Y$ is connected. The arguments of the proof of Lemma \ref{lneg} now apply (i.e. $Y$ may have finitely many isomorphism types, each having a hyperbolic fundamental group) and it follows that $\min_Y I(Y)\ge C'_\epsilon$ where $C'_\epsilon>0$ is a constant that only depends on $\epsilon$. Hence if (\ref{y}) holds we have $I(X,X')\ge \min_Y I(Y)\ge C'_\epsilon$ and the result follows.

Suppose that inequality (\ref{y}) is false, i.e. $I(X,X') < \min_Y I(Y)$, and consider a simplicial loop $\gamma:S^1\to X'$ satisfying $\gamma\sim 1$ in $X$ and
$|\gamma|\cdot A_X(\gamma)^{-1} <\min_Y I(Y).$ Let $\psi: D^2\to X$ be a simplicial spanning disc of minimal area.
It follows from the arguments of Ronan \cite{Ron}, that $\psi$ is non-degenerate in the following sense: for any 2-simplex $\sigma$ of $D^2$ the image $\psi(\sigma)$ is a 2-simplex
and for two distinct 2-simplexes $\sigma_1, \sigma_2$ of $D^2$ with $\psi(\sigma_1)=\psi(\sigma_2)$ the intersection $\sigma_1\cap \sigma_2$ is either $\emptyset$ or a vertex of $D^2$. In other words, we exclude {\it foldings}, i.e. situations such that $\psi(\sigma_1)=\psi(\sigma_2)$ and $\sigma_1\cap \sigma_2$ is an edge.
Consider $Z=X'\cup \psi(D^2)$. Note that $L(Z)\le 0$. Indeed, since $$L(Z)=\sum_e (2-\deg_Z(e)),$$ where $e$ runs over the edges of $Z$, we see that for $e\subset X'$,
$\deg_{X'}(e)\le \deg_Z(e)$ and for a newly created edge $e\subset \psi(D^2)$, clearly $\deg_Z(e)\ge 2$. Hence, $L(Z)\le L(X')\le 0$.
On the other hand, $A_X(\gamma)=A_Z(\gamma)$ and hence $I(Z)\le |\gamma|\cdot A_X(\gamma)^{-1}<\min_Y I(Y)$, a contradiction.
\end{proof}

The main idea of the proof of Theorem \ref{uniform} in the general case is to find a planar complex (a \lq\lq singular  surface\rq\rq)\ $\Sigma$, with one boundary component $\partial_+\Sigma$ being the initial loop and such that \lq\lq the rest of the boundary\rq\rq\, $\partial_-\Sigma$ is a \lq\lq product of negative loops\rq\rq\,  (i.e. loops satisfying Lemma \ref{lnegs}). The essential part of the proof is in estimating the area (the number of 2-simplexes) of such $\Sigma$.

\begin{proof}[Proof of Theorem \ref{uniform}]
Consider a connected tight pure $\epsilon$-admissible 2-complex $X$
and a simplicial prime loop $\gamma: S^1 \to X$ such that the ratio
$|\gamma|\cdot A_X(\gamma)^{-1}$ is less than the minimum of the numbers $I(X')$ for all proper subcomplexes $X'\subset X$. Consider a minimal spanning disc
$b: D^2\to X$ for $\gamma=b|_{\partial D^2}$; here $D^2$ is a triangulated disc and $b$ is a simplicial map. As we showed in Remark \ref{rmrk}, the map $b$ is surjective.
As explained in the proof of Lemma \ref{lnegs}, due to arguments of Ronan \cite{Ron},
we may assume that $b$ has no foldings.

For any integer $i\ge 1$ we denote by $X_i\subset X$ the pure subcomplex generated by all 2-simplexes $\sigma$ of $X$ such that the preimage $b^{-1}(\sigma)\subset D^2$ contains $\ge i$ two-dimensional simplexes. One has $X=X_1\supset X_2\supset X_3\supset \dots.$ Each $X_i$ may have several connected components and we will denote by $\Lambda$ the set labelling all the connected components of the disjoint union $\sqcup_{i\ge 1} X_i$. For $\lambda\in \Lambda$ the symbol $X_\lambda$ will denote the corresponding connected component of $\sqcup_{i\ge 1} X_i$
and the symbol
$i=i(\lambda)\in \{1, 2, \dots\}$ will denote the index $i\ge 1$ such that $X_\lambda$ is a connected component of $X_i$, viewed as a subset of $\sqcup_{i\ge 1} X_i$. We endow $\Lambda$ with the following partial order: $\lambda_1\le \lambda_2$ iff $X_{\lambda_1}\supset X_{\lambda_2}$ (where $X_{\lambda_1}$ and $X_{\lambda_2}$ are viewed as subsets of $X$) and $i(\lambda_1)\le i(\lambda_2)$.

Next we define the sets
$$\Lambda^-=\{\lambda\in \Lambda; L(X_\lambda)\le 0\}$$
and
$$ \Lambda^+=\{\lambda\in \Lambda; \mbox{for any $\mu\in \Lambda$ with $\mu\le \lambda$, }\, L(X_\mu)> 0\}.$$
Finally we consider the following subcomplex  of the disk $D^2$:
\begin{eqnarray}\label{defsigma}
\Sigma' =D^2-\bigcup_{\lambda\in \Lambda^-}{\rm {Int}}(b^{-1}(X_\lambda))\end{eqnarray}
and we shall denote by $\Sigma$ the connected component of $\Sigma'$ containing the boundary circle $\partial D^2$.

Recall that for a 2-complex $X$ the symbol $f_2(X)$ denotes the number of 2-simplexes in $X$. We have
\begin{eqnarray}\label{one1}
f_2(D^2) = \sum_{\lambda\in \Lambda} f_2(X_\lambda),
\end{eqnarray}
and
\begin{eqnarray}\label{two2}
f_2(\Sigma) \le f_2(\Sigma') = \sum_{\lambda\in \Lambda^+} f_2(X_\lambda).
\end{eqnarray}
Formula (\ref{one1}) follows from the observation that any 2-simplex of $X=b(D^2)$ contributes to the RHS of (\ref{one1}) as many units as its multiplicity (the number of its preimages under $b$). Formula (\ref{two2}) follows from (\ref{one1}) and from the fact that for a 2-simplex $\sigma$ of $\Sigma$ the image $b(\sigma)$ lies always in the complexes $X_\lambda$ with $L(X_\lambda)> 0$.

\begin{lemma} \label{lpartial} One has the following inequality
\begin{eqnarray}\label{llb}
\sum_{\lambda\in \Lambda^+}L(X_\lambda)\le |\partial D^2|.
\end{eqnarray}
\end{lemma}

See \cite{CF1}, Lemma 6.8 for the proof.

Now we continue with the proof of Theorem \ref{uniform}.
Consider a tight $\epsilon$-admissible pure 2-complex $X$ and a simplicial loop $\gamma: S^1\to X$ as above.
We will use the notation introduced earlier. The complex $\Sigma$ is a connected subcomplex of the disk $D^2$; it contains the boundary circle $\partial D^2$
which we will denote also by
$\partial_+\Sigma$. The closure of the complement of $\Sigma$,
$$N=\overline{D^2-\Sigma}\subset D^2$$ is a pure 2-complex. Let $N=\cup_{j\in J}N_j$ be the strongly connected components of $N$.
Each $N_j$ is PL-homeomorphic to a disc and we define
$$\partial_-\Sigma=\cup_{j\in J}\partial N_j,$$ the union of the circles $\partial N_j$ which are the boundaries of the strongly connected components of $N$.
It may happen that $\partial_+\Sigma$ and $\partial_-\Sigma$ have nonempty intersection. Also, the circles forming $\partial_-\Sigma$ may not be disjoint.

We claim that for any $j\in J$ there exists $\lambda\in \Lambda^-$ such that $b(\partial N_j)\subset X_\lambda$.
Indeed, let $\lambda_1, \dots, \lambda_r\in \Lambda^-$ be the minimal elements of $\Lambda^-$ with respect to the partial order introduced earlier. The complexes
$X_{\lambda_1}, \dots, X_{\lambda_r}$ are connected and pairwise disjoint and for any $\lambda\in \Lambda^-$ the complex $X_\lambda$ is a subcomplex
of one of the
sets $X_{\lambda_i}$, where $i=1, \dots, i$. From our definition (\ref{defsigma}) it follows that the image of the circle $b(\partial N_j)$ is contained in the union
$\cup_{i=1}^r X_{\lambda_i}$ but since $b(\partial N_j)$ is connected it must lie in one of the sets $X_{\lambda_i}$.

We may apply Lemma \ref{lnegs} to each of the circles $\partial N_j$. We obtain that each of the circles $\partial N_j$ admits a spanning discs of area
$\le K_\epsilon |\partial N_j|$, where $K_\epsilon= C'^{-1}_\epsilon$ is the inverse of the constant given by Lemma \ref{lnegs}.
Using the minimality of the disc $D^2$ we obtain that the circles $\partial N$ bound in $D^2$ several discs with
the total area
$A \le K_\epsilon\cdot |\partial_-\Sigma|.$

For $\lambda\in \Lambda^+$ one has $L(X_\lambda)\ge 1$ and $\chi(X_\lambda)\le 1$ (since $b_2(X_\lambda)=0$); in particular, $f_1(X_\lambda)\geq f_2(X_\lambda)$. By Lemma \ref{mu1 or mu2} either 
$$3\chi(X_\lambda) +L(X_\lambda) \ge \epsilon f_1(X_\lambda), \quad \mbox{or}\quad 2\chi(X_\lambda) +L(X_\lambda) \ge \epsilon f_2(X_\lambda).$$
Hence we have either
$$4L(X_\lambda) \ge 3\chi(X_\lambda) + L(X_\lambda) \ge \epsilon f_1(X_\lambda)\ge \epsilon f_2(X_\lambda) $$
or
$$3L(X_\lambda) \ge 2\chi(X_\lambda) + L(X_\lambda) \ge \epsilon f_2(X_\lambda).$$
Since $L(X_\lambda)\ge 1$ both cases imply
$$f_2(X_\lambda)\le \frac{3}{\epsilon} L(X_\lambda).$$
Summing up we get
$$f_2(\Sigma)\le  \sum_{\lambda\in \Lambda^+} f_2(X_{\lambda}) \le \frac{3}{\epsilon}\sum_{\lambda\in \Lambda^+}L(X_\lambda) \le
 \frac{3}{\epsilon}
 |\partial D^2|.$$
The rightmost inequality  is given by Lemma \ref{lpartial}.

Next we observe, that
\begin{eqnarray}
|\partial_-\Sigma| \le 2f_2(\Sigma) +|\partial_+\Sigma|.
\end{eqnarray}
Therefore, we obtain
\begin{eqnarray*}
f_2(D^2) &\le& f_2(\Sigma)+A \, \le \, \frac{3}{\epsilon} |\gamma| + K_\epsilon\cdot 2\cdot f_2(\Sigma) + K_\epsilon |\gamma| \\
&\le &
\left(\frac{3}{\epsilon}(1+2K_\epsilon) +K_\epsilon\right)\cdot |\gamma|,
\end{eqnarray*}
implying
\begin{eqnarray}
I(X) \ge \frac{\epsilon}{3+6K_\epsilon+\epsilon K_\epsilon}.
\end{eqnarray}
This completes the proof of Theorem \ref{uniform}.
\end{proof}

\bibliographystyle{amsalpha}

\vskip 0.3cm
A. Costa and M. Farber: 
\vskip 0.3cm
School of Mathematical Sciences, Queen Mary University of London, London E1 4NS
\vskip 0.3cm
 Email addresses:     \hskip 1cm A.Costa@qmul.ac.uk, \hskip 1cm M.Farber@qmul.ac.uk
\end{document}